\documentclass[12pt,a4paper]{amsart}

\usepackage{graphicx}
\usepackage{comment}
\usepackage{amsmath, amssymb, amsthm, amscd, amsxtra, mathtools}
\usepackage{latexsym}
\usepackage{epsfig}

\usepackage{mathrsfs}  % \mathscr
\usepackage[usenames,dvipsnames,svgnames,table]{xcolor}
\usepackage[linktocpage=true, colorlinks=true, linkcolor=Blue, citecolor=BrickRed, urlcolor=RoyalBlue]{hyperref}

\usepackage[alphabetic]{amsrefs}

\numberwithin{equation}{section} % Number equations by section

\newtheorem{theorem}{Theorem}
\newtheorem{lemma}[theorem]{Lemma}
\newtheorem{proposition}[theorem]{Proposition}
\newtheorem{corollary}[theorem]{Corollary}

\theoremstyle{definition}
\newtheorem{definition}[theorem]{Definition}

\theoremstyle{remark}

\def \div {\mathrm{div}}

\title[Keller-Osserman and Harnack type results for elliptic PDE]{Keller-Osserman and Harnack type results for nonlinear elliptic PDE with unbounded ingredients}

\author[B. Sirakov]{Boyan Sirakov}
\author[A. Sobral]{Aelson Sobral}
\subjclass{35B50, 35J15}

\begin{document}

\begin{abstract} 
We show that the classical Keller-Osserman theorem on the solvability of the equation 
$\mathcal{L}[u] = f(u)$
is valid when $\mathcal{L}$ is a general operator in divergence form with unbounded coefficients in the natural regime of local integrability. This has been open up to now, earlier results concerned operators with locally bounded ingredients. We also settle an open question from \cite{SS21} about the validity of the strong maximum principle for supersolutions of $\mathcal{L}[u] = f(u)$ under the optimal integral condition of V\'azquez. More generally, we obtain a  Harnack inequality for positive solutions of this equation, which extends a result by V. Julin. 
\end{abstract}

\maketitle

%%%%%%%%%%%%%%%%%
%\tableofcontents
%%%%%%%%%%%%%%%%%

\section{Introduction and main results}
\subsection{Brief overview of novelties}
We start by recalling the following very classical result on general solvability of semilinear elliptic PDE, obtained (independently) by Keller \cite{K57} and Osserman \cite{O57} in 1957: given a continuous nondecreasing function $f\colon[0,\infty)\to [0,\infty)$, the equation 
\begin{equation}\label{kelosseq}
\Delta u = f(u)
\end{equation} 
has a solution (and even a subsolution) $u\colon\mathbb{R}^n\to \mathbb{R}$ if and only if the primitive $F(t) = \int_0^t f(s)\,ds$ is such that
\begin{equation}\label{kelosscond}
 \int_{1}^\infty \frac{dt}{\sqrt{F(t)}} \quad\mbox{diverges}.
\end{equation} 
This result has generated a huge body of research, which continues to this day (a partial review of the literature will be given below). 

A very extensively studied question is the possibility of considering more general elliptic operators or more general nonlinear functions in \eqref{kelosseq}. One of our main results below states that the Keller-Osserman result holds if the Laplacian is replaced by a uniformly elliptic operator in divergence form with {\it (locally) unbounded coefficients}. This question has been completely open up to now, we settle it here for coefficients in uniformly local Lebesgue spaces. 
\smallskip

Another fundamental question on the equation \eqref{kelosseq} is the validity of the {\it strong minimum principle}, that is, whether a nonnegative (super)solution which attains the value zero vanishes identically. It is very classical that this is so if $f$ grows at most linearly around zero, but in 1984 V\'asquez \cite{V84} (see also Pucci-Serrin \cite{PS00}) showed that the right hypothesis for the validity of the SMP is as follows: if $u$ is a nonnegative supersolution, that is, $\Delta u \le f(u)$, $u\ge0$ in a domain $\Omega$, and $u(x_0)=0$ then $u\equiv 0$ in~$\Omega$ provided 
\begin{equation}\label{vascond}
 \int_{0}^1 \frac{dt}{\sqrt{F(t)}} \quad\mbox{diverges}.
\end{equation} 

For the validity of the SMP, the possibility of replacing the Laplacian by an operator with locally unbounded coefficients was considered for the first time only recently in~\cite{SS21}. The main result on SMP from that paper states that for such operators the SMP is valid provided 
\begin{equation}\label{sscond}
 \limsup_{t\searrow 0} \frac{f(t)}{t\,(\ln t)^2}>0.
\end{equation} 
The condition \eqref{sscond} is stronger than \eqref{vascond}, so the question of the validity of the SMP under the optimal \eqref{vascond} remained open (see \cite[p.4-5]{SS21}). We will answer this question in the affirmative here. 
\smallskip 

The validity of the SMP under the V\'asquez integral condition is a consequence of a result from \cite{SS21} and our third main result here, which is an extension of the Harnack inequality of Julin \cite{VJ1}. The generalized Harnack inequality in \cite{VJ1} states that any nonnegative solution of \eqref{kelosseq} (or of div$(A(x)Du)=f(u)$ for a bounded uniformly positive matrix $A(x)$) in $B_2$ is such that for some constant $C>0$ independent of $u$ and $f$ 
\begin{equation}\label{julin}
    \int_{\inf_{B_1} u}^{\sup_{B_1} u} \frac{dt}{\sqrt{F(t)} + t} \leq C.
\end{equation}
We will show that this inequality is valid for solutions of more general equations with unbounded coefficients.

\subsection{Precise definitions and results}

Let $\Omega \subset \mathbb{R}^n$ be an arbitrary domain and $\mathcal{L}$ be the linear uniformly elliptic second-order operator in general divergence form 
\begin{equation}\label{form of the eq}
    \mathcal{L}[u] = \div\left(A(x)Du + b_1(x)u\right) + b_2(x) \cdot Du + c(x)u.
\end{equation}
The measurable and bounded matrix-valued function $A(x)$  satisfies the uniform ellipticity condition  
$$
    \lambda |\xi|^2 \leq A(x)\xi \cdot \xi \leq \Lambda |\xi|^2 \quad \text{for all } \xi \in \mathbb{R}^n,
$$  
for almost every $x \in \Omega$, with constants $0 < \lambda \leq \Lambda$. The lower-order coefficients belong locally to Lebesgue spaces, which makes it possible for weak solutions to satisfy the maximum principle and the Harnack inequality; specifically,
\begin{equation}\label{integrability of exponents}
    b_1,b_2 \in L^q_{loc}(\Omega) \quad c \in L^{q/2}_{loc}(\Omega) \quad \text{with} \quad q > n.
\end{equation}  

We will consider solutions (resp. supersolutions, subsolutions) of 
\begin{equation}\label{maineq}
    \mathcal{L}[u] = (\mbox{resp. }\le, \ge) f(u)\quad \mbox{in }\; \Omega
\end{equation}
which belong to the Sobolev space $H^1_{\mathrm{loc}}(\Omega)$ and solve the inequality in the standard weak sense (see for instance \cite[Chapter 8]{GT}), of course always assuming that $f$ and $u$ are such that $f(u)\in H^{-1}(\Omega)$ (for instance $f(u)\in L^s_{\mathrm{loc}}(\Omega)$ for $s=2n/(n+2)$ if $n\ge 3$ or $s>1$ if $n=2$ is enough). Here and throughout the paper, $f\colon[0,\infty)\to [0,\infty)$ is a continuous nondecreasing function.

A most natural question on \eqref{maineq} is its solvability when $\Omega=\mathbb{R}^n$. As noted above, for $\mathcal{L}=\Delta$ and classical solutions, it was answered already in 1957. That result by Keller and Osserman was one of the first general solvability results for semilinear elliptic equations and has had an enormous influence over the years. 
 In particular, there has been a huge amount of work on extending the Keller-Osserman result to more general operators in \eqref{maineq} and weak solutions, with extensions to quasilinear, fully nonlinear, singular, degenerate elliptic operators. It has also generated a lot of work on the closely related question of explosive solutions in bounded domains, as well as a more detailed description of the behaviour of the solutions. We refer the interested reader to \cites{AGQ, CDLV14, CDLV16, DGGW, Fa, FSs, FS, FQS13,  FPRArma, FPR,  K13, Ra}, as well as to the citations in these works. 
 
 Even though Keller-Osserman type results are available for a large variety of classes of elliptic operators, as evidenced by these references, a common feature of previous results is the local boundedness of the spatial dependence of the coefficients in the equation. This can be explained by the the original Keller-Osserman and still most frequently used approach to this problem -- one uses a subsolution constructed by replacing the coefficients by their lower bounds and writing a one-dimensional version of the resulting autonomous equation (that is, searching for radial solutions of that equation), which can be solved by ODE methods and its explicit behaviour can be determined precisely under the integral hypothesis \eqref{kelosscond}. Obviously, such a method cannot be used for locally unbounded coefficients. Our first contribution here is to show that the Keller-Osserman result holds true for operators that have locally unbounded coefficients with finite uniformly local norms.
 
We recall the definition of uniformly local Lebesgue spaces and norms. We say that $h\in L^q_{ul}(\Omega)$, $1\le q \le\infty$,  provided the quantity (norm)
\begin{equation}\label{deful}
\|h\|_{L^q_{ul}(\Omega)} \coloneqq \sup_{x\in \mathbb{R}^n} \|h\|_{L^q(\Omega\cap B_1(x))}
\end{equation}
is finite. Uniformly local $L^q$-norms essentially measure the local integrability features of a function; they  have been often used in the study of global existence of solutions of Navier-Stokes or parabolic equations, e.g.~in the classical papers \cite{GV, Ka}. Interestingly, the main result from \cite{SS21} on the seemingly unrelated Landis conjecture also represented an extension of results on equations with bounded coefficients to coefficients in uniformly local Lebesgue spaces. These spaces also play a pivotal role in the recent work on optimized elliptic estimates \cite{SS24}.

Our Keller-Osserman type result is as follows. 
\begin{theorem}\label{Keller-Osserman condition} Assume that \begin{equation}\label{coefs-loc-unif-integrable}
    b_1,\, b_2 \in L^q_{ul}(\mathbb{R}^n) 
    \quad \text{and} \quad 
    c \in L^{q/2}_{ul}(\mathbb{R}^n),\qquad\mbox{for some }\;q>n.
\end{equation}
Let $u \in H^1(\mathbb{R}^n)$ be a continuous entire weak subsolution to \eqref{maineq} in $\mathbb{R}^n$.
 Then
 \begin{equation}\label{kellosres}
    \int_{1}^\infty \frac{dt}{\sqrt{F}(t)} = \infty.
\end{equation}
\end{theorem}
 We recall that even for $\mathcal{L}=\Delta$ it is easy to construct a classical entire positive solution of $\mathcal{L}[u] = f(u)$, if $f$ does not satisfy \eqref{kellosres}.
 \medskip

Our research was initially triggered by a different question: under what conditions does the strong maximum principle (SMP) hold for nonnegative supersolutions of \eqref{maineq}? This problem dates back to 1984, when V\'azquez showed in \cite{V84} that classical nonnegative solutions to
\begin{equation}\label{vasquez MP}
    \Delta u \leq f(u) \quad \text{in} \quad \Omega,
\end{equation}
satisfy the SMP in $\Omega$ provided $f$ behaves like a modulus of continuity, meaning that it is nonnegative, nondecreasing,  $f(0) = 0$, and $f$  meets the additional condition
\begin{equation}\label{PS condition}
     \int_{0}\frac{dt}{\sqrt{F(t)}} = \infty \quad \text{with} \quad F(t) \coloneqq \int_{0}^{t}f(s) \, ds.
\end{equation}
This condition was later shown in \cite{PS00} to be both necessary and sufficient for the validity of the SMP. 
While \eqref{PS condition} allows for nonlinearities which are Lipschitz at zero, such cases fall within the scope of the standard SMP (see \cites{EV, H52,GT}); so \eqref{PS condition} becomes crucial only when dealing with nonlinearities that grow more quickly than Lipschitz at zero. Many subsequent works have extended the V\'azquez SMP to broader settings in terms of the differential operators considered and the dependence of $f$ in $x$, $u$ and $Du$; we refer to \cite{FMQ09, PR18, PS00, PSZ},  and the book \cite{PS} which is mostly dedicated to this topic (many more references are available in these works).  

For operators with unbounded coefficients, the theory remains significantly less developed, since many of the techniques used in the bounded-coefficient setting do not readily extend to the unbounded case. Actually, in the latter setting, the question of the validity of the SMP for a superlinear nonlinearity $f(u)$ had been completely open until the recent work \cite{SS21}, which introduced a novel approach based on an optimized weak Harnack inequality. More specifically, the proof of the SMP in \cite{SS21} relies on an application of the weak Harnack inequality in which the constant is optimized with respect to the norms of the lower order coefficients of the ``linearized" version of \eqref{maineq}. This method bridged variational and non-variational theories, enabling progress beyond the Lipschitz threshold for $f$ when $\mathcal{L}$ has unbounded coefficients; in particular, it allowed for nonlinearities of the form
$$
    f(s) \approx s|\ln(s)|^a, \quad \text{for} \quad a \leq 2.
$$
The case where the right-hand side grows even more rapidly at zero while still satisfying \eqref{PS condition} was stated as an open question in \cite{SS21}. At that time, the answer seemed unclear since \eqref{PS condition} is an ODE-type hypothesis, while, as we explained above, for equations with unbounded coefficients, the usual method of constructing a radial subsolution of the equation (a barrier) in a well-chosen annulus does not work.

On the other hand, a promising direction arises from considering ODEs for the sub- and super-level sets of $u$, as opposed to barrier construction. In \cite{VJ1}, V. Julin explored this idea for an operator in divergence form involving only the principal second-order part. Specifically, in this important advance of the theory of semilinear elliptic PDE he proved the generalized Harnack inequality \eqref{genharn} below  for solutions to  
\begin{equation}\label{julineq}
    \div\left(A(x)Du \right) = f(u),
\end{equation}
by deriving an ODE for the level sets of $u$ that incorporates the right-hand side behavior. This novel for \eqref{julineq} approach is inspired by techniques from the Harnack inequality proof in \cite{DBTR}, and yields a quantified version of the SMP.

The extension of the ODE-related technique from \cite{VJ1} to scenarios involving lower-order coefficients is  delicate, especially in regard to inequalities of Caccioppoli-type. These inequalities must be carefully managed to control the growth of energy terms relative to the size of sub- and super-level sets (this is precisely where the assumption in \eqref{integrability of exponents} plays a critical role and demonstrates its sharpness). Through iterations of these inequalities, we arrive at the following result.

\begin{theorem}\label{main theorem}
Let $u \in H^1(B_2)$ be a nonnegative solution to  
$$
    \mathcal{L}[u] = f(u).
$$  
There exists a constant $C = C(n, \lambda, \Lambda, q, \|b_1\|_{L^q(B_2)}, \|b_2\|_{L^q(B_2)}, \|c\|_{L^{q/2}(B_2)})$ such that  
\begin{equation}\label{genharn}
    \int_{\inf_{B_1} u}^{\sup_{B_1} u} \frac{dt}{\sqrt{F(t)} + t} \leq C.
\end{equation}
\end{theorem}

The overall scheme of the proof of Theorem \ref{main theorem} is similar to that of the main result in \cite{VJ1}, so we outline it only briefly. First, as in \cite{VJ1}  we establish the differential inequality for supersolutions
$$
    \eta'(t) \geq c \min\left\{\frac{1}{\sqrt{F(t)}}\eta(t)^{\frac{n-1}{n}}, \frac{1}{t}\eta(t) \right\}, \quad \text{for} \quad \eta(t) = |\{x \in \Omega \colon u(x) < t \}|,
$$
along with an analogous estimate for the measure of sublevel sets. Then we use this inequality to obtain the bound
\begin{equation}\label{etaest}
    \eta(t) \lesssim
  \left(\frac{t}{\sqrt{F(t)} + t} \right)^\frac{q n}{q-n},
\end{equation}
at a certain point, which extends the corresponding estimate from \cite{VJ1}, where \eqref{etaest} was proved without the power $qn/(q-n)$, for \eqref{julineq}. The estimate \eqref{etaest} enables us, through an iterative application of a Caccioppoli inequality, to refine our control over the infimum, thereby estimating part of the integral in \eqref{genharn}. To handle the remaining portion of the integral, we apply the differential inequality for the sublevel sets, which ensures that the solution remains bounded over a set of large measure. Then, on the subsolution side, we derive a second Caccioppoli inequality, which provides control over the supremum.

As we noted above, the main challenge in extending this approach to the unbounded setting lies in deriving Caccioppoli estimates that exhibit a sufficiently strong decay of the energy in relation to the size of the sub- and superlevel sets. Notably, the condition \eqref{integrability of exponents} plays a sharp role in achieving this, at least within the scope of our technique.

As a consequence of the generalized Harnack principle \eqref{genharn}, we establish the strong maximum principle (SMP) for \eqref{maineq}. For this, we also need a result from \cite{SS21} to deal with the fact that the proof of Theorem \ref{main theorem} relies on information from both sub- and supersolutions, while the SMP is stated only for supersolutions. 
 
\begin{theorem}\label{maximum principle}
Let $u \in H^1(\Omega)$ be a nonnegative solution to  
$$
    \mathcal{L}[u] \leq f(u) \quad \text{in} \quad \Omega,\qquad\mbox{and} 
\quad
    \int_{0}\frac{dt}{\sqrt{F(t)}+t} = \infty.
$$  
Then $u > 0$ or $u \equiv 0$ in $\Omega$.
\end{theorem}

The paper is organized as follows. We dedicate Section \ref{prelim-sct} to recalling some basic facts and known results. In Section \ref{supersolution-sct}, we prove estimates for nonnegative supersolutions, as well as the ODEs for the level sets. In Section \ref{subsolution-sct}, we seek estimates for subsolutions, and, in particular, prove a Cacciopoli inequality. The proof of the Keller-Osserman condition, Theorem \ref{Keller-Osserman condition}, is done in subsection \ref{subsct keller-osserman}, the generalized Harnack inequality, Theorem \ref{main theorem}, is done in Subsection \ref{main thm-subsct}, and the strong maximum principle in Subsection \ref{max princ-subsct}.

\section{Preliminaries}\label{prelim-sct}

In this section, we recall some well-known definitions and results. We use the same notation as in \cite{VJ1} and refer to that paper for more details on the background, as well as more references. 

Given a symmetric matrix $A \in \mbox{Sym}(n)$, $\mbox{Spec}(A)$ denotes the spectrum of $A$, that is, the set of eigenvalues of $A$. For a set $E \subseteq \mathbb{R}^n$, we denote by $\mathcal{X}_{E}$  the indicator function of  $E$. With the letter $\Omega \subset \mathbb{R}^n$, we denote a domain where a given equation is set. We denote with $\|\cdot\|_q$ the $L^q$-norm of a function when the domain of integration is clear from the context.  Solutions to \eqref{maineq} are understood in the usual weak sense given by the following definition.

\begin{definition}\label{definition of soolution}
We say that $u\in H^1_{\mathrm{loc}}(\Omega)$ satisfies $\mathcal{L}[u] \leq f(u)$ in $\Omega$ in the weak sense (supersolution) if for every $x_0 \in \Omega$ and nonnegative $w \in H^1_0(B_r(x_0))$, with $B_r(x_0) \subset \Omega$, we have
$$
    -\int_{B_r(x_0)} A(x)Du \cdot Dw + ub_1(x)\cdot Dw + \int_{B_r(x_0)} wb_2(x)\cdot Du + c(x)u\,w \leq \int_{B_r(x_0)}w\,f(u).
$$
Likewise, $u$ satisfies $\mathcal{L}[u] \geq f(u)$ in $\Omega$ in the weak sense (subsolution) if for every $x_0 \in \Omega$ and nonnegative $w \in H^1_0(B_r(x_0))$, with $B_r(x_0) \subset \Omega$, we have
$$
    -\int_{B_r(x_0)} A(x)Du \cdot Dw + ub_1(x)\cdot Dw + \int_{B_r(x_0)} wb_2(x)\cdot Du + c(x)u\,w \geq \int_{B_r(x_0)}w\,f(u).
$$
\end{definition}
By smoothing the operator's coefficients, we can replicate Julin's approximation argument (see \cite{VJ1}*{Page 883}) and assume that the solution and all the coefficients of $\mathcal{L}$ are smooth. Let us also define 
$$
    E_{t} \coloneqq \{x \in \Omega \colon u(x) < t\}, \quad \text{and} \quad A_t \coloneqq \{x \in \Omega \colon u(x) \geq t \},
$$
the sublevel and superlevel sets of a function $u$, respectively.

We recall the concept of perimeter, which plays a central role in deriving the ODE for the sub- and super-level sets.

\begin{definition}\label{def of perimeter}
Let $U \subset \mathbb{R}^n$ be an open set and $E \subset \mathbb{R}^n$. The perimeter of $E$ in $U$ is defined as
$$
    P(E,U) \coloneqq \sup\left\{\int_E \mbox{div}\, \varphi \colon \varphi \in C^1_0(U,\mathbb{R}^n),\, \|\varphi\|_{L^\infty(U)} \leq 1 \right\}.
$$
\end{definition}

Next, we recall the fundamental relative isoperimetric inequality.

\begin{lemma}\label{relativ isoperimetric inequality}
Let $E \subset \mathbb{R}^n$ be a set of finite perimeter in $B_r$. Then
$$
    P(E,B_r) \geq c \min\left\{|E \cap B_r|^{\frac{n-1}{n}},|B_r \backslash E|^{\frac{n-1}{n}} \right\},
$$
for a dimensional constant $c>0$.
\end{lemma}

The following estimate is a simple consequence of H\"older's inequality, but will be handy when obtaining the ODE for the super and sub-level sets.

\begin{lemma}\label{ineq for the perimeter}
Let $u \in C^\infty(B_r(x_0))$ be such that $|Du|>0$ in $\{u = t\}$, then
$$
    \left(P(E_t,B_r(x_0))\right)^2 \leq \|Du\|_{L^1(\{u = t\} \cap B_r(x_0))}\, \left\|\frac{1}{|Du|}\right\|_{L^1(\{u = t\}\cap B_r(x_0))}.
$$
\end{lemma}

\begin{proof}
Let $\varphi$ be as in the definition of perimeter. Then, by the Divergence Theorem, we have
$$
    \int_{E_t} \mbox{div}\, \varphi\,dx = \int_{\partial^* E_t} \varphi \cdot \nu\, d\mathcal{H}^{n-1},
$$
where $\partial^* E_t$ denotes the reduced boundary of $E_t$. Since $\partial^* E_t \subseteq \{u=t\}$, we obtain $\nu = Du/|Du|$, and so
$$
    \int_{E_t} \mbox{div}\, \varphi \,dx = \int_{\partial^* E_t} \varphi \cdot \frac{Du}{|Du|}\, d\mathcal{H}^{n-1} = \int_{\partial^* E_t \cap B_r(x_0)} \varphi \cdot \frac{Du}{|Du|^{\frac{1}{2}}} \frac{1}{|Du|^{\frac{1}{2}}}\, d\mathcal{H}^{n-1},
$$
where we used that $\varphi$ is supported in $B_r(x_0)$. By the H\"older inequality we have, since  $\|\varphi\|_{L^\infty(B_r(x_0))} \leq 1$,
\begin{eqnarray*}
    \int_{E_t} \mbox{div}\, \varphi \, dx & \leq & \left\{\int_{\partial^* E_t \cap B_r(x_0)}|\varphi|^2 |Du| \, d\mathcal{H}^{n-1} \right\}^{\frac{1}{2}} \left\{\int_{\partial^* E_t \cap B_r(x_0)} \frac{1}{|Du|}\, d\mathcal{H}^{n-1}\right\}^{\frac{1}{2}}\\
    & \leq & \left\{\int_{\partial^* E_t \cap B_r(x_0)} |Du| \, d\mathcal{H}^{n-1}\right\}^{\frac{1}{2}} \left\{\int_{\partial^* E_t \cap B_r(x_0)} \frac{1}{|Du|}\, d\mathcal{H}^{n-1}\right\}^{\frac{1}{2}}\\
    & = & \|Du\|_{L^1(\partial^* E_t \cap B_r(x_0))}^{\frac{1}{2}}\, \left\|\frac{1}{|Du|} \right\|_{L^1(\partial^* E_t \cap B_r(x_0))}^\frac{1}{2},
\end{eqnarray*}
and the proof is done by taking the supremum in $\varphi$, since $\partial^* E_t \subseteq \{u=t\}$.
\end{proof}

We will also need the following real-analysis result, whose proof can be found in \cite{EG}*{Lemma 7.1}.
\begin{lemma}\label{useful real analysis result}
Let $b>0$ and $(x_i)_{i \in \mathbb{N}}$ be a sequence of real positive numbers such that
$$
    x_{i+1} \leq CB^ix_i^{1+b},
$$
with $C>0$ and $B>1$. If $x_0 \leq C^{-\frac{1}{b}}B^{-\frac{1}{b^2}}$, then 
$$
    \lim_{i \to \infty}x_i = 0.
$$
\end{lemma}

\section{Refined estimates for supersolutions}\label{supersolution-sct}
In this section, we obtain estimates for weak solutions of
$$
    \mathcal{L}[u] \leq f(u).
$$
We recall that hereafter $\mathcal{L}$ has the form \eqref{form of the eq} and the lower order coefficients satisfy hypothesis \eqref{integrability of exponents}.

The first result is a Caccioppoli-type inequality, establishing a decay estimate essential for deriving the ODE governing the sub- and superlevel sets. The presence of unbounded terms in the operator slows the decay rate, precisely reflecting the integrability of the lower-order terms.

\begin{proposition}\label{cacciopoli inequality}
Assume $u \in C^\infty(B_R)$ is a nonnegative supersolution of
$$
    \mathcal{L}[u] \leq f(u).
$$
For any $0 < r < R$, we have the inequality
\begin{eqnarray}
    C^{-1}\|Du\|_{L^2(E_t \cap B_r)} & \leq & \left(\sqrt{F(t)} + \frac{t}{R-r} \right)|E_t \cap B_R|^{\frac{1}{2}}  \label{estimate in L^2} \\ 
    & + &  \frac{t}{\sqrt{R-r}}|E_t \cap B_R|^{\frac{1}{2}-\frac{1}{2q}}+  t|E_t \cap B_R|^{\frac{1}{2}-\frac{1}{q}}. \nonumber
\end{eqnarray}
Furthermore, for almost every $t>0$, there  holds
\begin{eqnarray*}
    C^{-1}\int_{\{u=t \} \cap B_r} |Du|  &\leq & \left(\frac{F(t)}{t}+\frac{\sqrt{F(t)}}{R-r} +\frac{t}{(R-r)^2} \right)|E_t \cap B_R| + t|E_t \cap B_R|^{1-\frac{2}{q}}\\
    & + & \left(\sqrt{F(t)} +\frac{t}{R-r} \right)|E_t \cap B_R|^{1 - \frac{1}{q}} + \frac{t}{\sqrt{R-r}}|E_t \cap B_R|^{1 - \frac{3}{2q}}\\
    & + & \frac{t}{(R-r)^{3/2}}|E_t \cap B_R|^{1-\frac{1}{2q}},
\end{eqnarray*}
where $C = C(n,\lambda,\Lambda,\|b_1\|_{q},\|b_2\|_{q},\|c\|_{q/2})$.
\end{proposition}

\begin{proof}
Let $\zeta \colon \mathbb{R}^n \to [0,1]$ be a smooth cutoff that satisfies
$$
    \zeta = 1 \quad \mbox{in} \quad B_r, \quad \zeta = 0 \quad \mbox{in} \quad B_R \quad \mbox{and} \quad \|D\zeta\|_{L^\infty(B_R)} \leq 2(R-r)^{-1},  
$$
and let us consider $\varphi = (t-u)_+ \zeta^2$ as a test function. Computing
$$
    D\varphi = -\zeta^2\, \mathcal{X}_{E_t}Du + 2(t-u)_+ \zeta\, D\zeta,
$$
we obtain,
\begin{eqnarray*}
        -\int_{E_t\cap B_R} \varphi f(u) & \leq & -\int_{E_t\cap B_R} \zeta^2 A(x)Du \cdot Du + 2\int_{E_t\cap B_R} \zeta\, (t-u)_+ A(x)Du \cdot D\zeta\\
        & & + \int_{E_t \cap B_R}u b_1 \cdot D\varphi - \int_{E_t\cap B_R} \varphi b_2 \cdot Du - \int_{E_t \cap B_R}c\,u \varphi.
\end{eqnarray*}
Since $\mbox{Spec}(A(x)) \subset [\lambda,\Lambda]$, we have $A(x)Du \cdot Du \geq \lambda |Du|^2$ and $|A(x)Du|\leq \Lambda |Du|$. Substituting this in the previous inequality, we have
\begin{eqnarray*}
        \lambda \int_{E_t\cap B_R} \zeta^2\, |Du|^2  & \leq & \int_{E_t\cap B_R} \varphi f(u) + 2\Lambda\int_{E_t\cap B_R} \zeta\, (t-u)_+ |Du|\, |D\zeta|\\
        & &  + \int_{E_t \cap B_R}u |b_1|\, |D\varphi| + \int_{E_t\cap B_R} \varphi |b_2|\,|Du| + \int_{E_t \cap B_R}|c|u \varphi. 
\end{eqnarray*}
By Fubini's theorem, we have 
\begin{eqnarray*}
    \int_{B_R} f(u)(t-u)_+ & = & \int_{B_R}f(u)\int_{0}^{t}\mathcal{X}_{\{u < s \}}ds\, dx\\
    & = & \int_{0}^{t} \int_{B_R} f(u)\mathcal{X}_{\{u < s \}}dx\, ds\\
    & \leq & \int_{0}^{t}f(s) \int_{B_R} \mathcal{X}_{\{u < t \}}dx\, ds = F(t)|E_t \cap B_R|,
\end{eqnarray*}
where 
$$
    F(t) \coloneqq \int_{0}^tf(s).
$$
We recall Young's inequality for products: given  nonnegative real numbers $a$ and $b$, 
\begin{equation}\label{young inequality}
    ab \leq \frac{a^2}{2\epsilon} + \frac{\epsilon b^2}{2}
\end{equation}
for any $\epsilon>0$. Therefore, the following inequality holds
\begin{eqnarray*}
    2\Lambda\int_{E_t \cap B_R} \zeta (t-u)_+|Du||D\zeta| &\leq& \Lambda\int_{E_t \cap B_R} \epsilon |Du|^2\zeta^2 + \frac{\Lambda}{\epsilon}\int_{E_t \cap B_R}|D\zeta|^2(t-u)_+^2\\
    \int_{E_t \cap B_R}u |b_1|\, |D\varphi| & \leq & \frac{4t^2}{R-r}\|b_1\|_{L^1} + \frac{\epsilon}{2}
    \int_{E_t \cap B_R}|Du|^2 \zeta^4 + \frac{t^2}{2\epsilon}\|b_1\|_{L^2}^2\\
    \int_{E_t \cap B_R}(t-u)_+\zeta^2|b_2|\,|Du|& \leq & \frac{\epsilon}{2} \int_{E_t \cap B_R}|Du|^2\zeta^4 + \frac{t^2}{2\epsilon}\|b_2\|_{L^2}^2\\
    \int_{E_t \cap B_R}|c|u \varphi & \leq & t^2\|c\|_{L^{q/2}}|E_t \cap B_R|^{1 - \frac{2}{q}},
\end{eqnarray*}
where we intentionally omitted the domain of integration $(E_t \cap B_R)$ in the norms  to ease notation. Picking $\epsilon = \lambda(2(\Lambda+1))^{-1}$, so that $\left(\lambda-\epsilon(\Lambda+1)\right) = \lambda/2$, we obtain
\begin{eqnarray*}
    \frac{\lambda}{2} \int_{E_t\cap B_R} \zeta^2\, |Du|^2  & \leq & F(t)|E_t \cap B_R| + \frac{\Lambda t^2}{\epsilon(R-r)^2}|E_t \cap B_R| + \frac{C_1t^2}{R-r}|E_t \cap B_R|^{1 - \frac{1}{q}}\\
    & + & \frac{t^2}{4\epsilon}\left(C_1|E_t \cap B_R|^{1 - \frac{2}{q}} + C_2 |E_t \cap B_R|^{1 - \frac{2}{q}} + C_3|E_t \cap B_R|^{1 - \frac{2}{q}}\right),
\end{eqnarray*}
for $C_1 = C_1(\|b_1\|_{q})$, $C_2 = C_2(\|b_2\|_{q})$, $C_3 = C_3(\|c\|_{q/2})$. Then we are done once we recall that $\zeta = 1$ in $B_r$, which gives
\begin{eqnarray*}
    C^{-1}\|Du\|_{L^2(E_t \cap B_r)} & \leq & \left(\sqrt{F(t)} + \frac{t}{R-r} \right)|E_t \cap B_R|^{\frac{1}{2}} + \frac{t}{\sqrt{R-r}}|E_t \cap B_R|^{\frac{1}{2}-\frac{1}{2q}}\\
    & + &  t|E_t \cap B_R|^{\frac{1}{2}-\frac{1}{q}},
\end{eqnarray*}
for a constant $C = C(n,\lambda,\Lambda,\|b_1\|_{q},\|b_2\|_{q},\|c\|_{q/2})$. 
\medskip

To prove the second inequality, we consider the values $t>0$ such that $|Du|>0$ on $\{u = t\}$ (we recall that by the Morse-Sard Lemma this happens for almost every $t>0$, since $u$ is smooth). Given a parameter $\rho >0$, we integrate the equation in the set $E_t \cap B_\rho$ (recall we are assuming that the coefficients are smooth) to obtain
\begin{equation}\label{L1 grad integral}
    \int_{E_t \cap B_\rho} \mbox{div}\left(A(x)Du + b_1\, u \right) + b_2 \cdot Du + c\, u \leq \int_{E_t \cap B_\rho} f(u).
\end{equation}
By the divergence theorem, we have
\begin{eqnarray*}
    \mathcal{I} \coloneqq \int_{E_t \cap B_\rho} \mbox{div}\left(A(x)Du + b_1\,u \right) & = & \int_{\partial^* \left(E_t \cap B_\rho\right)} (A(x)Du + b_1\, u)\cdot \nu d\mathcal{H}^{n-1},
\end{eqnarray*}
where $\nu$ is the normal vector defined on the manifold $\partial^* (E_t \cap B_\rho)$. Inputting into \eqref{L1 grad integral} the fact that $\nu = Du/|Du|$ on the portion $\{u=t\} \cap B_\rho$, we have
\begin{eqnarray*}
    \int_{\{u=t\}  \cap B_\rho} A(x)Du \cdot \frac{Du}{|Du|} &\leq& \int_{E_t \cap B_\rho} f(u) + |b_2||Du| + |c|u - \int_{E_t \cap \partial B_\rho} A(x) Du \cdot \nu\\
    & + & \int_{\partial^* \left(E_t \cap B_\rho\right)} |b_1||u|.
\end{eqnarray*}
Again using that $\text{Spec}(A(x)) \subset [\lambda,\Lambda]$, we have
\begin{eqnarray*}
    \lambda \int_{\{u=t\}  \cap B_\rho} |Du|  &\leq& \int_{E_t \cap B_\rho} f(u) + |b_2||Du| + |c|u + \Lambda \int_{E_t \cap \partial B_\rho} |Du|\\
    & + & \int_{\partial^* \left(E_t \cap B_\rho\right)} |b_1||u|.
\end{eqnarray*}
We integrate the inequality with respect to $\rho$ from $r$ to $\hat{\rho} = \frac{r+R}{2}$ in order to obtain
\begin{eqnarray}
    \lambda\left(\frac{R-r}{2} \right) \int_{\{u=t \}  \cap B_r} |Du|  &\leq & \left(\frac{R-r}{2} \right)\int_{E_t \cap B_{\hat{\rho}}} f(u) + |b_2||Du| + |c|u \label{first inequality for grad}\\
    & + & \Lambda \int_{E_t \cap (B_{\hat{\rho}}\backslash B_r) } |Du| + \int_{r}^{\hat{\rho}}\int_{\partial^* \left(E_t \cap B_{\rho}\right)} |b_1||u|. \nonumber
\end{eqnarray}
Let us estimate the terms involving $b_2$ and $b_1$. An application of the H\"older inequality and \eqref{estimate in L^2} gives
\begin{eqnarray*}
    \int_{E_t \cap B_{\hat{\rho}}}|b_2||Du|
    & \leq & \|b_2\|_{L^2}\, \|Du\|_{L^{2}}\\
    & \leq & C \|b_2\|_{L^{q}}|E_t \cap B_R|^{\frac{1}{2} - \frac{1}{q}}  \,\left(\sqrt{F(t)} +\frac{t}{R-r} \right)|E_t \cap B_R|^{\frac{1}{2}}\\
    & + &  C \|b_2\|_{L^{q}}|E_t \cap B_R|^{\frac{1}{2} - \frac{1}{q}}\frac{t}{\sqrt{R-r}}|E_t \cap B_R|^{\frac{1}{2} - \frac{1}{2q}}\\
    & + &  C \|b_2\|_{L^{q}}|E_t \cap B_R|^{\frac{1}{2} - \frac{1}{q}}t|E_t \cap B_R|^{\frac{1}{2} - \frac{1}{q}},
\end{eqnarray*}
and as a consequence,
\begin{eqnarray*}
    C_1^{-1}\int_{E_t \cap B_{\hat{\rho}}}|b_2||Du| & \leq &   \left(\sqrt{F(t)} +\frac{t}{R-r} \right)|E_t \cap B_R|^{1 - \frac{1}{q}} + t|E_t \cap B_R|^{1 - \frac{2}{q}}\\
    & + &  \frac{t}{\sqrt{R-r}}|E_t \cap B_R|^{1 - \frac{3}{2q}}.
\end{eqnarray*}
Another application of the H\"older inequality and the Trace Theorem (see \cite{EV}*{Section 5.5, Theorem 1}), gives
\begin{eqnarray*}
    \int_{\partial^* \left(E_t \cap B_{\rho}\right)} |b_1||u| & \leq & \|b_1\|_{q} \left\{\int_{\partial^* \left(E_t \cap B_{\rho}\right)} |u|^{q'}\right\}^{\frac{1}{q'}}\\
    & \leq & C\, \|b_1\|_{q} \left(\|u\|_{L^{q'}(E_t \cap B_{\rho})} + \|Du\|_{L^{q'}(E_t \cap B_{\rho})}\right),
\end{eqnarray*}
where $q'$ is the conjugate exponent of $q$. We now make use of the nonnegativity of $u$, H\"older's inequality and \eqref{estimate in L^2} to obtain
\begin{eqnarray*}
    \int_{\partial^* \left(E_t \cap B_{\rho}\right)} |b_1||u| & \leq & \overline{C} \left(t|E_t \cap B_\rho|^{1 - \frac{1}{q}} + \|Du\|_{L^2(E_t \cap B_\rho)}|E_t \cap B_\rho|^{1 - \frac{1}{q} - \frac{1}{2}} \right)\\
    & \leq & \overline{C}t|E_t \cap B_\rho|^{1 - \frac{1}{q}} + \overline{C}C\left(\sqrt{F(t)} + \frac{t}{R-r} \right)|E_t \cap B_\rho|^{1 - \frac{1}{q}}\\
    & & +  \overline{C}C\left(\frac{t}{\sqrt{R-r}}|E_t \cap B_\rho|^{1 - \frac{3}{2q}} + t|E_t \cap B_\rho|^{1 - \frac{2}{q}} \right),
\end{eqnarray*}
for some constant $\overline{C}$ depending on $n$, $q$ and $\|b_1\|_{q}$. Putting all together into inequality \eqref{first inequality for grad}, we obtain
\begin{eqnarray*}
    C_2^{-1}\int_{\{u=t \}  \cap B_r} |Du|  &\leq& \int_{E_t \cap B_R}f(u)+  \left(\sqrt{F(t)} +\frac{t}{R-r} \right)|E_t \cap B_R|^{1 - \frac{1}{q}}\\
    & + & \frac{t}{\sqrt{R-r}}|E_t \cap B_R|^{1 - \frac{1}{2q}-\frac{1}{q}} + t|E_t \cap B_R|^{1-\frac{2}{q}}\\
    & + & t|E_t \cap B_R|^{1-\frac{2}{q}} + \left(\frac{\sqrt{F(t)}}{R-r} + \frac{t}{(R-r)^2} \right)|E_t \cap B_R|\\
    & + & \frac{t}{(R-r)^{3/2}}|E_t \cap B_R|^{1-\frac{1}{2q}}+\frac{t}{R-r}|E_t \cap B_R|^{1-\frac{1}{q}}\\
    & + & t|E_t \cap B_\rho|^{1 - \frac{1}{q}} + |E_t \cap B_\rho|^{1 - \frac{1}{q}}\left(\sqrt{F(t)} + \frac{t}{R-r} \right)\\
    & + & \frac{t}{\sqrt{R-r}}|E_t \cap B_\rho|^{1 - \frac{3}{2q}} + t|E_t \cap B_\rho|^{1 - \frac{2}{q}},
\end{eqnarray*}
for a constant $C_2 = C_2(n,\lambda,\Lambda,\|b_1\|_{q},\|b_2\|_{q},\|c\|_{q/2})$. Simplifying this further gives
\begin{eqnarray*}
    C_3^{-1}\int_{\{u=t \}  \cap B_r} |Du|  &\leq& \int_{E_t \cap B_R}f(u) + t|E_t \cap B_R|^{1-\frac{2}{q}}\\
    & + &\left(\sqrt{F(t)} +\frac{t}{R-r} \right)|E_t \cap B_R|^{1 - \frac{1}{q}} + \frac{t}{\sqrt{R-r}}|E_t \cap B_R|^{1 - \frac{3}{2q}}\\
    & + &\left(\frac{\sqrt{F(t)}}{R-r} + \frac{t}{(R-r)^2} \right)|E_t \cap B_R| + \frac{t}{(R-r)^{3/2}}|E_t \cap B_R|^{1-\frac{1}{2q}},
\end{eqnarray*}
for $C_3 = 2C_2$. The estimate is complete once one notices that
$$
    \int_{E_t \cap B_R} f(u) \leq \frac{F(t)}{t}|E_t \cap B_R|.
$$
\end{proof}

Now we use the second part of the Cacciopoli-type inequality to get a partial ODE involving the sub- and superlevel sets.

\begin{lemma}\label{partial ODE}
Assume $u \in C^{\infty}(B_2)$ is a nonnegative supersolution of
$$
    \mathcal{L}[u] \leq f(u).
$$
Then, for almost every $t>0$, we have
$$
    \int_{\{u=t\}\cap B_{2r}(x_0)} \frac{1}{|Du|}d\mathcal{H}^{n-1} \geq \frac{c}{t}\, \min \big\{|E_t \cap B_r(x_0)|, |B_r(x_0) \backslash E_t| \big\},
$$
whenever $B_{2r}(x_0) \subset B_2$ is such that
\begin{equation}\label{relation between radius and t for partial ode}
    r \leq \frac{t}{\sqrt{F(t)}}.
\end{equation}
\end{lemma}

\begin{proof}
By considering $v(x) = u(x_0 + x)$, we may assume $x_0 = 0$. We split the proof into two cases depending on where the minimum is attained. Assume first that
\begin{equation}\label{min is the first one}
    \min \left\{|E_t \cap B_r|, |B_r \backslash E_t| \right\} = |E_t \cap B_r|. 
\end{equation}
Observe that the integrand on the LHS of the statement of the lemma is not defined along the set of critical points of $u$. However, Morse-Sard Lemma ensures that for almost every $t>0$ we have $|Du|>0$ on the level set $\{u=t\}\cap B_2$.

The proof now proceeds depending on whether the density of $E_t \cap B_r$ in $B_r$ is universally controlled away from zero and from one. Similar to the proof of \cite{VJ1}*{Lemma 3.3}, we split the proof further into two cases. Assume
\begin{equation}\label{first case}
    |E_t \cap B_r| \geq 2^{-(n+2)}|B_r|.
\end{equation}
By Proposition \ref{cacciopoli inequality} with $R=2r$, we have
\begin{eqnarray*}
    C^{-1}\,\|Du\|_{L^1(\{u=t \}  \cap B_r)}  &\leq&  \left(\frac{F(t)}{t}+\frac{\sqrt{F(t)}}{r} +\frac{t}{r^2} \right)|E_t \cap B_{2r}|+ t|E_t \cap B_{2r}|^{1-\frac{2}{q}}\\
    & + &  \left(\sqrt{F(t)} +\frac{t}{r} \right)|E_t \cap B_{2r}|^{1 - \frac{1}{q}} +  \frac{t}{\sqrt{r}}|E_t \cap B_{2r}|^{1 - \frac{3}{2q}}\\
    & + & \frac{t}{r^{3/2}}|E_t \cap B_{2r}|^{1 - \frac{1}{2q}}.
\end{eqnarray*}
From \eqref{relation between radius and t for partial ode}, we obtain
$$
    \frac{F(t)}{t} \leq \frac{t}{r^2},
$$
and so
\begin{eqnarray*}
    C_1^{-1}\,\|Du\|_{L^1(\{u=t \}  \cap B_r)}  & \leq &  tr^{n-2} + tr^{n\left(1-\frac{2}{q}\right)}+tr^{n\left(1-\frac{1}{q}\right) - 1}+ tr^{n\left(1-\frac{3}{2q}\right) - \frac{1}{2}}\\
    &+& tr^{n\left(1-\frac{1}{2q}\right) - \frac{3}{2}}.
\end{eqnarray*}
Now, since  $q \geq n$, we have
$$
    C_2^{-1}\,\|Du\|_{L^1(\{u=t \}  \cap B_r)} \leq tr^{n-2}.
$$
By Lemma \ref{ineq for the perimeter}, we get
$$
   (P(E_t,B_r))^2\, \left(\int_{\{u=t\}\cap B_r}\frac{1}{|Du|} \right)^{-1} \leq \|Du\|_{L^1(\{u=t \}\cap B_r)}. 
$$
By Lemma \ref{relativ isoperimetric inequality} and \eqref{min is the first one}, we obtain
$$
    P(E_t,B_r) \geq c \min\left\{|E_t \cap B_r|^{\frac{n-1}{n}},|B_r \backslash E_t|^{\frac{n-1}{n}} \right\} = c |E_t \cap B_r|^{\frac{n-1}{n}}.
$$
Thus,
$$
   c^2 |E_t \cap B_r|^{\frac{2(n-1)}{n}} \left(\int_{\{u=t\}\cap B_r}\frac{1}{|Du|} \right)^{-1} \leq \|Du\|_{L^1(\{u=t \}\cap B_r)} \leq C\, tr^{n-2}.
$$
We now use \eqref{first case} to estimate $|E_t \cap B_r|$ from below and get
$$
    \int_{\{u=t\}\cap B_r}\frac{1}{|Du|}  \geq C_1 \frac{1}{t}r^n \geq C_2 \frac{1}{t}|E_t \cap B_r|.
$$
If \eqref{first case} is not true, then we have $|E_t \cap B_r| < 2^{-(n+2)}|B_r|$. In this case, we cover $E_t \cap B_r$ by a Besicovitch-type family of balls $\left\{\overline{B}_{R_i}(x_i)\right\}_{x_i \in E_t}$ for 
$$
    R_i \coloneqq \inf\left\{ \rho > 0 \colon |B_\rho(x_i)\cap E_t|\leq \left(1-2^{-n-1}\right)|B_\rho(x_i)|\right\}.
$$
We can then apply the first part of the proof (when we assumed \eqref{first case}), with $B_{R_i}(x_i)$ instead of $B_r$, and the result follows by a covering argument as in \cite{VJ1}*{Proof of Lemma 3.3}.

Now, if \eqref{min is the first one} is not true, that is
\begin{equation*}
    \min \left\{|E_t \cap B_r(x_0)|, |B_r(x_0) \backslash E_t| \right\} = |B_r(x_0) \backslash E_t|,
\end{equation*}
then we repeat the same proof depending on whether $|B_r \backslash E_t| \geq 2^{-(n+2)}|B_r|$ holds. The only difference in comparison to the first part is that $P(E_t, B_r) \geq c|B_r \backslash E_t|^{\frac{n-1}{n}}$.
\end{proof}

We now derive the ODE for the level sets. Since the argument closely follows that of \cite{VJ1}*{Proposition 3.5}, we omit the details. The differential inequality arising from the equation is used exclusively to establish Lemma \ref{partial ODE}, whose conclusion mirrors that of \cite{VJ1}. From this point onward, the proof proceeds using tools from measure theory.

\begin{proposition}\label{ODE for level sets}
Assume $u \in C^\infty(B_2)$ is a nonnegative supersolution of
$$
    \mathcal{L}[u] \leq f(u).
$$
Denote $\mu(t) = |A_t \cap B_2|$ and $\eta(t) = |E_t \cap B_2|$. Then, for almost every $t>0$ such that $\mu(t) \leq |B_2|/2$, we have 
$$
    -\mu'(t) \geq c\, \min\left\{\frac{1}{\sqrt{F(t)}}\mu(t)^{\frac{n-1}{n}}, \frac{1}{t}\mu(t) \right\},
$$
and for almost every $t>0$ with $\eta(t) \leq |B_2|/2$, 
$$
    \eta'(t) \geq c\, \min\left\{\frac{1}{\sqrt{F(t)}}\eta(t)^{\frac{n-1}{n}}, \frac{1}{t}\eta(t) \right\}.
$$
\end{proposition}

We conclude this section with a consequence of an iteration of Proposition \ref{cacciopoli inequality}, which is particularly useful when the sublevel sets exhibit sufficiently fast decay. In this case, geometric information about supersolutions can be deduced. The proof follows the classical approach used in establishing the Harnack inequality. This is another key point where the fact that the integrability of the lower-order coefficients exceeds the dimension becomes crucial.

\begin{lemma}\label{sublevel measure decay}
Assume $u \in C^\infty(B_2)$ is a nonnegative supersolution of
$$
    \mathcal{L}[u] \leq f(u).
$$
There is $\delta_0 = \delta_0(n,\lambda,\Lambda,\|b_1\|_{q},\|b_2\|_{q},\|c\|_{q/2})$  such that if for some $t >0$ 
$$
    |E_t \cap B_2| \leq \delta_0 \left(\frac{t}{\sqrt{F(t)} + t} \right)^\frac{q n}{q-n},
$$
then
$$
    \inf_{B_1} u \geq \frac{t}{2}.
$$
\end{lemma}
\begin{proof}
First, we pick $\delta_0 \leq \frac{1}{2}|B_2|$. Observe that
$$
    |E_t \cap B_2| \leq \delta_0 \left(\frac{t}{\sqrt{F(t)} + t} \right)^\frac{q n}{q-n} \leq \delta_0.
$$
Now, given $0 < h < k \leq t$, we define $v$ as 
\[   
v = 
     \begin{cases}
       k &\quad \text{if} \quad  u \geq k\\
       k-u &\quad\text{if} \quad  h < u < k \\
       k-h &\quad\text{if} \quad  u \leq h.\\
     \end{cases}
\]
By Gagliardo-Niremberg-Sobolev inequality for $v$ we get
\begin{eqnarray*}
    (k-h)|E_h \cap B_\rho|^{1 - \frac{1}{n}} & \leq & \|v\|_{L^{\frac{n}{n-1}}(E_k \cap B_\rho)}\\
    & \leq & c\|Dv\|_{L^1(E_k \cap B_\rho)}\leq c|E_k \cap B_\rho|^{\frac{1}{2}}\|Du\|_{L^2(E_k \cap B_\rho)} 
\end{eqnarray*}
for a dimensional constant $c$. By Proposition \ref{cacciopoli inequality}, for every $0 < \rho < R \leq 2$ we have
\begin{eqnarray*}
     C^{-1}\, \|Du\|_{L^2(E_k\cap B_\rho)}  &\leq&  \left(\sqrt{F(k)} +\frac{k}{R-\rho} \right)|E_k \cap B_R|^{\frac{1}{2}} + \frac{k}{\sqrt{R-r}}|E_k\cap B_R|^{\frac{1}{2}-\frac{1}{2q}}\\
     & + & k|E_k \cap B_R|^{\frac{1}{2}-\frac{1}{q}}. 
\end{eqnarray*}
Therefore,
\begin{eqnarray*}
    (4cC)^{-1}(k-h)|E_h \cap B_\rho|^{1 - \frac{1}{n}} &\leq&  \left(\sqrt{F(k)} + \frac{k}{R-\rho} \right)|E_k \cap B_R|^{1 - \frac{1}{q}}.
\end{eqnarray*}
Define $r_i = 1 + 2^{-i}$, $k_i = \frac{t}{2}(1 + 2^{-i})$ and apply the above inequality with $R = r_i$, $\rho = r_{i+1}$, $h = k_{i+1}$ and $k = k_i$. Set $x_i = |E_{k_i}\cap B_{r_i}|$; then 
$$
    (4cC)^{-1}\frac{t}{2^{i+2}}x_{i+1}^{\frac{n-1}{n}} \leq \left(\sqrt{F(t)} + \frac{t}{2^{-(i+1)}} \right)x_i^{1 - \frac{1}{q}}.
$$
Rearranging terms, we get
$$
    (4cC)^{-1}\,x_{i+1}^{\frac{n-1}{n}} \leq 16\left(\frac{\sqrt{F(t)}+t}{t}\right)4^i\,x_i^{1 - \frac{1}{q}},
$$
and so, since $x_i \geq x_{i+1}$, we obtain
$$
    (4cC)^{-1}\,x_{i+1} \leq 16\left(\frac{\sqrt{F(t)}+t}{t} \right)4^i\,x_i^{1 - \frac{1}{q} + \frac{1}{n}},
$$
which in turn can be written as
$$
    x_{i+1} \leq C_1\left(\frac{\sqrt{F(t)}+t}{t} \right)4^i x_i^{1+s} \quad \mbox{for} \quad s = \frac{q-n}{q n},
$$
and a constant $C_1 = C_1(c,C)$. Recall that by assumption $q > n$, so $s>0$. If $\delta_0$ is chosen small enough such that
$$
    \delta_0 \leq C_1^{-\frac{1}{s}}4^{-\frac{1}{s^2}},
$$
then
$$
    x_0 \coloneqq |E_t \cap B_2| \leq \delta_0 \left(\frac{t}{\sqrt{F(t)} + t} \right)^\frac{q n}{q-n}\leq \left(C_1\left(\frac{\sqrt{F(t)}+t}{t} \right)\right)^{-\frac{1}{s}}4^{-\frac{1}{s^2}}.
$$
We can now apply Lemma \ref{useful real analysis result} to obtain
$$
    \lim_{i \to \infty} x_i = 0,
$$
which means that $\inf_{B_1}u \geq t/2$.
\end{proof}

\section{Estimates for subsolutions}\label{subsolution-sct}

In this section, we derive estimates for subsolutions. The main result is the following Caccioppoli-type estimate, which extends the one established in \cite{VJ1}.

\begin{lemma}\label{di giorgi iteration}
Let $u \in H^1(B_{2r}(x_0))$ be a subsolution of
$$
    \mathcal{L}[u] \geq f(u),
$$
and denote $M_r = \sup_{B_r(x_0)}u$. Assume further that there is $\eta>0$ such that
\begin{equation}\label{coefs universally small}
    \|b_1\|_{L^{q}(B_{2r}(x_0))} + \|b_2\|_{L^{q}(B_{2r}(x_0))} + \|c\|_{L^{q/2}(B_{2r}(x_0))} \leq \eta.
\end{equation}
Then there are constants $\beta, \gamma \in (0,1)$ depending on $n,q$ such that for every $\rho < r$ and $t < M_r$ it holds
\begin{eqnarray*}
    c_2^{-1} \left(\frac{M_r^2}{F(t)} \right)^{\beta-1} \int_{B_\rho(x_0)}(u-t)_+^2 & \leq & \frac{1}{(r-\rho)^2}|A_t \cap B_r(x_0)|^{\gamma} \int_{B_r(x_0)} (u-t)_+^2 \\
    & + & \left(\frac{t+M_r}{r-\rho}\right)|A_t \cap B_r(x_0)|^{\frac{1}{2} + \gamma}\left(\int_{B_r(x_0)}(u-t)^2_+ \right)^{\frac{1}{2}}\\
    & + & t^2|A_t \cap B_r(x_0)|^{1 + \gamma},
\end{eqnarray*}
for $c_2 = c_2(n,\lambda,\Lambda)>0$, provided $\eta=\eta(n,\lambda,\Lambda)$ is chosen small enough.
\end{lemma}

Before delivering the proof, a few comments are in place. In \cite{VJ1}, the author proved a simplified estimate with $\beta = 1/2$ and $\gamma = 1/n$, and the second and third elements of the right-hand side of the inequality were nonexistent. They appear here due to the lower-order coefficients, which significantly slow down the decay. In our case, choosing suitable constants $\beta$ and $\gamma$ (that depend upon each other) is crucial to ensure a decay of the same order. We emphasize that the integrability of the lower-order coefficients being greater than the dimension is also key in finding those parameters.

\begin{proof}
With no loss of generality, we may assume $x_0 = 0$. We use as test function $\varphi \coloneqq (u-t)_+\zeta^2$, where $\zeta \colon \mathbb{R}^n \to [0,1]$ is a smooth cutoff which satisfies
$$
    \zeta = 1 \quad \mbox{in} \quad B_\rho, \quad \zeta = 0 \quad \mbox{in} \quad B_r \quad \mbox{and} \quad \|D\zeta\|_{L^\infty(B_r)} \leq 2(r-\rho)^{-1}.  
$$
We obtain
\begin{eqnarray*}
    \int_{B_r}|D\left((u-t)_+\zeta \right)|^2 &\leq& C \int_{B_r} (u-t)_+^2|D\zeta|^2 - \int_{B_r} \varphi f(u) + \int_{B_r} \zeta^2 (u-t)_+ b_2 \cdot Du\\
    & + &  \int_{B_r}cu \varphi - ub_1 \cdot D\varphi, 
\end{eqnarray*}
for $C = C(\lambda,\Lambda)$. Let us estimate the terms involving $c$ and $b_1$ first. We repeat the very same arguments of \cite{Lin}*{Theorem 4.1} to obtain
\begin{eqnarray*}
    \int_{B_r}cu\varphi & \leq &  \|c\|_{L^{q/2}(B_r)}\left(|A_t\cap B_r|^{\frac{2}{n}-\frac{2}{q}}\int_{B_r}|D((u-t)_+\zeta)|^2 + t^2|A_t \cap B_r|^{1-\frac{2}{q}} \right).
\end{eqnarray*}
To estimate the $b_1$ term, we proceed as follows
\begin{eqnarray*}
    \int_{B_r} ub_1 \cdot D\varphi & = & \int_{B_r} tb_1 \cdot D\varphi + \int_{B_r} (u-t)_+b_1 \cdot D\varphi \\
    & = & \int_{B_r}2 t\zeta (u-t)_+ b_1 \cdot D\zeta + \int_{B_r} t\zeta^2 b_1 \cdot D(u-t)_+\\
    & + & \int_{B_r} 2\zeta (u-t)^2_+b_1 \cdot D\zeta + \int_{B_r} \zeta^2(u-t)_+ b_1 \cdot D(u-t)_+\\
    & = & I + II + III + IV.
\end{eqnarray*}
We estimate terms $I$ and $III$ using the H\"older inequality which leads to 
$$
    |I| + |III| \leq 4\|b_1\|_{q}\left(\frac{t+M_r}{r-\rho}\right)|A_t \cap B_r|^{\frac{1}{2}-\frac{1}{q}} \left(\int_{B_r}(u-t)_+^2\right)^{\frac{1}{2}}.
$$
By Young's inequality \eqref{young inequality}, we obtain
\begin{eqnarray*}
    |II| & \leq & \int_{B_r}|t\zeta b_1|\,|\zeta D(u-t)_+|\\
    & \leq & \frac{1}{2\epsilon}t^2\|b_1\|^2_{q}|A_t \cap B_r|^{1 - \frac{2}{q}} + \frac{\epsilon}{2} \int_{B_r}|\zeta D(u-t)_+|^2,
\end{eqnarray*}
for any $\epsilon > 0$. The term $IV$ can be incorporated into the estimate involving the term $b_2$. By monotonicity of $f$ and for $x \in A_t \cap B_r$, we have
$$
    f(u(x)) \geq f(t) \geq \frac{F(t)}{t} \geq \frac{F(t)}{M_r^2}(u(x)-t),
$$
which readily implies
$$
    \int_{B_r} f(u)\varphi \geq \frac{F(t)}{M_r^2}\int_{B_r}(u-t)_+^2\zeta^2.
$$
Putting these estimates together with the equation information, we obtain
\begin{eqnarray*}
    \int_{B_r}|D\left((u-t)_+\zeta \right)|^2& \leq& \frac{C}{(r-\rho)^2} \int_{B_r} (u-t)_+^2 - \frac{F(t)}{M_r^2}\int_{B_r}(u-t)_+^2\zeta^2\\
    & +& \int_{B_r} \zeta^2 (u-t)_+ (b_1 + b_2) \cdot Du + t^2\|c\|_{L^{q/2}(B_r)} |A_t \cap B_r|^{1-\frac{2}{q}}\\
    & + & \|c\|_{q/2}|A_t\cap B_r|^{\frac{2}{n}-\frac{2}{q}}\int_{B_r}|D((u-t)_+\zeta)|^2\\
    & + & 4\left(\frac{t+M_r}{r-\rho}\right)\|b_1\|_{q}|A_t \cap B_r|^{\frac{1}{2} - \frac{1}{q}}\left(\int_{B_r}(u-t)^2_+ \right)^{\frac{1}{2}}\\
    & + & \frac{1}{2\epsilon}t^2\|b_1\|^2_{q}|A_t \cap B_r|^{1 - \frac{2}{q}} + \frac{\epsilon}{2}\int_{B_r}|\zeta D(u-t)_+|^2.
\end{eqnarray*}
Now, let us concentrate on the third term in the right-hand side of the previous inequality. Define $b \coloneqq b_1 + b_2$. For $\epsilon > 0$, Young's inequality \eqref{young inequality} gives
\begin{eqnarray*}
    \int_{B_r} \zeta^2 (u-t)_+b \cdot Du & \leq & \int_{B_r} \zeta^2 (u-t)_+\,|b| \, |Du|\\
    & = &  \int_{B_r}  \zeta(u-t)_+\,|b| \, |\zeta\, D(u-t)_+|\\
    & \leq & \frac{\epsilon}{2} \int_{B_r}  |\zeta\, D(u-t)_+|^2 + \frac{1}{2\epsilon} \int_{B_r}  \zeta^2(u-t)^2_+\,|b|^2\\
    & \leq & \frac{16\epsilon}{(r-\rho)^2}\int_{B_r} (u-t)_+^2 + 4\epsilon \int_{B_r} |D(\zeta\, (u-t)_+)|^2\\
    & + &  \frac{1}{4\epsilon} \int_{B_r}  \zeta^2(u-t)^2_+\,|b|^2,
\end{eqnarray*}
where we have used the relation
$$
    D(\zeta(u-t)_+) = \zeta D(u-t)_+ + (u-t)_+D\zeta.
$$
Therefore,
\begin{eqnarray*}
    C_0\int_{B_r}|D\left((u-t)_+\zeta \right)|^2& \leq& \frac{C}{(r-\rho)^2} \int_{B_r} (u-t)_+^2 - \frac{F(t)}{M_r^2}\int_{B_r}(u-t)_+^2\zeta^2\\
    & +& \frac{16\epsilon}{(r-\rho)^2}\int_{B_r} (u-t)_+^2 + \frac{1}{2\epsilon} \int_{B_r}  \zeta^2(u-t)^2_+\,|b|^2\\
    & + & C(\epsilon,c,b_1)t^2 |A_t \cap B_r|^{1 - \frac{2}{q}}\\
    & + & 4\left(\frac{t+M_r}{r-\rho}\right)\|b_1\|_{q}|A_t \cap B_r|^{\frac{1}{2} - \frac{1}{q}}\left(\int_{B_r}(u-t)^2_+ \right)^{\frac{1}{2}},
\end{eqnarray*}
for $C_0 = (1-4\epsilon-\|c\|_{q/2})$. Denoting $w = (u-t)_+ \zeta$, we observe that by the H\"older inequality we have
\begin{eqnarray*}
    \int_{B_r}  w^2\,|b|^2 & \leq & \|b\|_{L^{q}(B_r)}^2\,\|w\|_{L^{q'}(B_r)}^2,
\end{eqnarray*}
for $q' = \frac{2q}{q-2}$. Since $q>n$, we have $q' < 2^\star$, where $2^\star$ is the critical Sobolev exponent. By the Gagliardo-Nirenberg-Sobolev inequality, we have
\begin{eqnarray*}
    \int_{B_r}  w^2\,|b|^2 & \leq & C^\star \|b\|_{L^{q}(B_r)}^2\,\|Dw\|_{L^2(B_r)}^2,
\end{eqnarray*}
for a dimensional constant $C^\star$. Thus, we obtain
\begin{eqnarray}\label{ineq in cacciop}
    C'\int_{B_r}|D\left((u-t)_+\zeta \right)|^2& \leq& \frac{1}{(r-\rho)^2} \int_{B_r} (u-t)_+^2 - \frac{F(t)}{M_r^2}\int_{B_r}(u-t)_+^2\zeta^2\\
    & + & t^2 |A_t \cap B_r|^{1-\frac{2}{q}} \nonumber \\ 
    & + & \left(\frac{t+M_r}{r-\rho}\right)|A_t \cap B_r|^{\frac{1}{2} - \frac{1}{q}}\left(\int_{B_r}(u-t)^2_+ \right)^{\frac{1}{2}}, \nonumber
\end{eqnarray}
where
$$
    C' = \left(1-4\epsilon-\|c\|_{q/2} - \frac{C^\star \|b\|_{L^{q}(B_r)}^2}{4\epsilon}\right) C_1^{-1}, \quad \mbox{for} \quad C_1 = C_1(n,\lambda,\Lambda,\|b_1\|_{q},\|c\|_{q/2}).
$$
Picking $\epsilon = \frac{1}{36}$ and assuming, by \eqref{coefs universally small}, $\|c\|_{q/2} + C^\star\|b\|_{L^{q}(B_r)}^2 \leq 1/36$, we have $C' \geq 1/2$. By the H\"older and Sobolev inequalities, we have
\begin{equation}\label{holder and sobolev in cacciopoli}
    \int_{B_r}w^2 \leq \left(\int_{B_r}w^{2^\star} \right)^{\frac{2}{2^{\star}}}|A_t \cap B_r|^{1 - \frac{2}{2^\star}} \leq C_3|A_t \cap B_r|^{\frac{2}{n}}\int_{B_r}|Dw|^2,
\end{equation}
for some dimensional constant $C_3$. Now we let $\beta \in (0,1)$ be a positive number satisfying 
$$
    \beta q > n
$$
which is possible since $q > n$. For this choice, it holds
$$
    1 - \frac{2}{q} + \frac{2\beta}{n} > 1.  
$$
We use the following generalized Young's inequality: given any nonnegative real numbers $a$ and $b$, 
$$
    ab \leq \frac{a^{p_1}}{p_1} + \frac{b^{p_1'}}{p_1'} \quad \text{for} \quad \frac{1}{p_1'} = 1 - \frac{1}{p_1}.
$$
An application of this inequality with $p_1 = 1/\beta$ gives
\begin{equation}\label{gen young trick}
     \frac{F(t)}{M_r^2} + \frac{c^{-1}}{2}|A_t \cap B_r|^{\frac{-2}{n}} \geq \frac{c_1^{-1/\beta}}{2\beta(1-\beta)} \left(\frac{F(t)}{M_r^2} \right)^{1-\beta}|A_t \cap B_r|^{-\frac{2\beta}{n}}.
\end{equation}
We use inequalities \eqref{holder and sobolev in cacciopoli} and \eqref{gen young trick} to further estimate \eqref{ineq in cacciop}, obtaining
$$
     c_2^{-1} \left(\frac{F(t)}{M_r^2} \right)^{1-\beta}|A_t \cap B_r|^{\frac{-2\beta}{n}}\int_{B_r}(u-t)_+^2\zeta^2 \leq 
$$
\begin{eqnarray*}
    & & \frac{1}{(r-\rho)^2} \int_{B_r} (u-t)_+^2 + \left(\frac{t+M_r}{r-\rho}\right)|A_t \cap B_r|^{\frac{1}{2} - \frac{1}{q}}\left(\int_{B_r}(u-t)^2_+ \right)^{\frac{1}{2}}\\
    & & +  t^2 |A_t \cap B_r|^{1-\frac{2}{q}}.
\end{eqnarray*}
Rearranging terms and using that $\zeta = 1$ in $B_{\rho}$ we get
\begin{eqnarray*}
    c_2^{-1}\int_{B_\rho}(u-t)_+^2 & \leq & \frac{1}{(r-\rho)^2}|A_t \cap B_r|^{\frac{2\beta}{n}} \left(\frac{M_r^2}{F(t)} \right)^{1-\beta} \int_{B_r} (u-t)_+^2\\
    & + & t^2|A_t \cap B_r|^{1 - \frac{2}{q}+\frac{2\beta}{n}}\left(\frac{M_r^2}{F(t)} \right)^{1-\beta}\\
    & + & \left(\frac{M_r^2}{F(t)} \right)^{1-\beta}\left(\frac{t+M_r}{r-\rho}\right)|A_t \cap B_r|^{\frac{1}{2} - \frac{1}{q} + \frac{2\beta}{n}}\left(\int_{B_r}(u-t)^2_+ \right)^{\frac{1}{2}}.
\end{eqnarray*}
\end{proof}

The iteration of this result leads to the following lemma.

\begin{lemma}\label{est from below}
Let $u \in H^1(B_{2r}(x_0))$ be a subsolution to
$$
    \mathcal{L}[u] \geq f(u),
$$
and denote $M_r \coloneqq \sup_{B_r(x_0)}u$. Assume \eqref{coefs universally small} holds. There exist a constant $c_0 = c_0(n,\lambda,\Lambda)>0$ such that if $0 < M_r \leq 2u(x_0)$, then
$$
    \frac{u(x_0)}{\sqrt{F\left(\frac{u(x_0)}{2} \right)}} \geq c_0r.
$$
\end{lemma}

\begin{proof}
Define $v(x) \coloneqq u(x_0 + rx)$, for $x \in B_2$. In this setting, $0 < M_1 \coloneqq \sup_{B_1}v \leq 2v(0)$ and we have
$$
    \mathcal{L}^r[v] \geq r^2f(v),
$$
where
$$
    \mathcal{L}^r[v] = \mbox{div}(A(x_0 + rx)Dv + rb_1(x_0 + rx)v) + rb_2(x_0+rx)\cdot Dv + r^2c(x_0+rx)v.
$$
In terms of $v$, we need to show that
$$
    \Lambda_0 \coloneqq \left(\frac{v(0)}{\sqrt{r^2 F\left(\frac{v(0)}{2} \right)}}\right)^2 \geq c_0^2 > 0,
$$
for some universal constant $c_0$. Let $\tau, \rho$ be such that $1/2 < \rho < \tau < 1$ and $h,k$ such that $v(0)/2 \leq h < k \leq v(0)$. We use that $M_1 \leq 2v(0)$ and Lemma \ref{di giorgi iteration} to obtain
\begin{eqnarray*}
    4^{\beta-1}c_2^{-1}\Lambda_0^{\beta-1}\int_{A_k \cap B_\rho}(v-k)_+^2 & \leq & \frac{1}{(\tau-\rho)^2}|A_k \cap B_\tau|^{\gamma} \int_{A_h \cap B_\tau} (v-h)_+^2 \\
    & + & 2\left(\frac{h+M_1}{\tau-\rho}\right)|A_k \cap B_\tau|^{\frac{1}{2} + \gamma}\left(\int_{A_h \cap B_\tau}(v-h)^2_+ \right)^{\frac{1}{2}}\\
    & + &  4h^2|A_k \cap B_\tau|^{1 + \gamma},
\end{eqnarray*}
for constants $\beta, \gamma \in (0,1)$. Taking into account that
$$
    |A_k \cap B_\tau| \leq \frac{1}{(k-h)^2}\int_{A_h \cap B_\tau}(v-h)_+^2,
$$
the following inequality holds
$$
    4^{-2}c_2^{-1}\int_{A_k \cap B_\rho}(v-k)_+^2 \leq
$$
$$
     \Lambda_0^{1-\beta}\left(\frac{1}{(\tau-\rho)^2} + \frac{h+M_1}{(\tau-\rho)(k-h)}+\frac{h^2}{(k-h)^2}\right)\frac{1}{(k-h)^{2\gamma}}\left(\int_{A_h \cap B_\tau}(v-h)_+^2\right)^{1+\gamma}.
$$
Defining $r_i = 2^{-1}(1 + 2^{-i})$ and $k_i = v(0)(3/4 - 4^{-1-i})$ we can write the previous inequality for $h = k_i$, $k = k_{i+1}$, $\tau = r_i$ and $\rho = r_{i+1}$, thus obtaining
$$
    4^{-2}c_2^{-1}\left(\int_{A_{k_i} \cap B_{r_i}}(v-k_i)_+^2\right)^{-(1+\gamma)}\int_{A_{k_{i+1}} \cap B_{r_{i+1}}}(v-k_{i+1})_+^2 \leq
$$    
$$    
     \Lambda_0^{1-\beta}\left(\frac{1}{(r_{i+1}-r_i)^2} + \frac{k_i+M_1}{(r_i-r_{i+1})(k_{i+1}-k_i)} + \frac{2k_i^2}{(k_{i+1}-k_i)^2}\right)\frac{1}{(k_{i+1}-k_i)^{2\gamma}}.
$$
Setting
$$
    x_i = M_1^{-2}\int_{A_{k_i}\cap B_{r_i}}(v -k_i)^2_+,
$$
and taking into account the relations $k_{i+1} - k_i = 4^{-(i+1)}v(0)$, and $r_i - r_{i+1} = 2^{-(i+2)}$, we have the following recurrence relation
$$
    x_{i+1} \leq C\Lambda_0^{1-\beta}4^i\left(\frac{M_1}{v(0)}\right)^{2\gamma} x_i^{1 + \gamma},
$$
and so, using $M_1 \leq 2v(0)$ once more, we obtain
$$
    x_{i+1} \leq C\Lambda_0^{1-\beta}4^i x_i^{1 + \gamma},
$$
for a large constant $C$. Now, notice that it should hold
$$
    \Lambda_0^{1-\beta} \geq |B_1|^{-\frac{1}{n}}C^{-1}4^{-n},
$$
otherwise, we would have
$$
    |B_1| < \Lambda_0^{-n(1-\beta)}C^{-n}4^{-n^2}.
$$
and we would obtain
\begin{eqnarray*}
    x_0 &\coloneqq& M_1^{-2} \int_{A_{\frac{v(0)}{2}} \cap B_1}\left(v-\frac{v(0)}{2}\right)^2_+ \leq  |B_1|\\
        & \leq & \Lambda_0^{-n(1-\beta)}C^{-n}4^{-n^2}.
\end{eqnarray*}
This would then imply by Lemma \ref{useful real analysis result} that $x_i \to 0$ as $i \to \infty$, meaning that
$$
    \sup_{B_{1/2}}v \leq \frac{3}{4}v(0),
$$
which is a contradiction. Thus we have
$$
    \Lambda_0 \geq c_0^2 \coloneqq \left(|B_1|^{-\frac{1}{n}}C^{-1}4^{-n} \right)^{\frac{1}{1-\beta}}.
$$
\end{proof}

Finally, taking advantage of the sign of the right-hand side, we prove the following result, which is a consequence of the local maximum principle.
\begin{lemma}\label{local max princ consequence}
Let $u \in H^1(B_{2r}(x_0))$ be a subsolution of
$$
    \mathcal{L}[u] \geq 0,
$$
and define $M_r \coloneqq \sup_{B_r(x_0)}u$. There exists a positive small constant $\epsilon_0$ such that if
$$
    \left|\left\{u \geq u(x_0)/2 \right\}\cap B_r(x_0)\right| \leq \epsilon_0|B_r(x_0)|,
$$
then $M_r \geq 4u(x_0)$, provided \eqref{coefs universally small} holds for some small $\eta=\eta(n,\lambda,\Lambda)>0$.
\end{lemma}

\begin{proof}
This is a consequence of the classical local maximum principle. Define $v = (u - u(x_0)/2)_+$ and observe it solves
$$
    \mathcal{L}[v] \geq g\,\mathcal{X}_{\{v>0 \}} \quad \mbox{for} \quad g = -\frac{u(x_0)}{2}\left(\mbox{div}(b_1) + c\right).
$$
By the local maximum principle (see \cite{SS21}*{Theorem 2.1}), we obtain
\begin{eqnarray*}
        \sup_{B_{r/2}(x_0)}v & \leq & C\left(\|v\|_{L^1(B_r(x_0))} + u(x_0)\left(r^{2 - \frac{2n}{q}}\|c\|_{L^{q/2}(B_{2r}(x_0))} + r^{2 - n/q}\|b_1\|_{L^q(B_{2r}(x_0))}\right) \right)\\
                    & \leq & C\left( \epsilon_0\left(M_r - \frac{u(x_0)}{2}\right) + u(x_0)r\left(\|c\|_{L^{q/2}(B_{2r}(x_0))} + \|b_1\|_{L^q(B_{2r}(x_0))}\right) \right).
\end{eqnarray*}
As a consequence,
$$
    u(x_0)\left(1 + C\epsilon_0  - Cr\left(\|c\|_{L^{q/2}(B_{2r}(x_0))} + \|b_1\|_{L^q(B_{2r}(x_0))}\right) \right) \leq (1-C\epsilon_0) M_r.
$$
Picking $\epsilon_0$ small enough and using \eqref{coefs universally small} for some small $\eta$, we obtain $M_r \geq 4u(x_0)$.
\end{proof}

\section{Proof of main results}

In this section, we gather the results from the previous sections and add the developments that yield the main theorems. 

\subsection{On the Keller-Osserman condition}\label{subsct keller-osserman}

We prove the following lower bound for subsolutions.

\begin{proposition}\label{integral lower bound}
Let $u \in H^1(B_{2R}(x_0))$ be a continuous and nonnegative subsolution of
$$
    \mathcal{L}[u] \geq f(u),
$$
and assume \eqref{coefs universally small} holds with $r = R$. There is a positive constant $c = c(n,\lambda,\Lambda)$, such that if $u(x_0)>0$, then
$$
    \int_{u(x_0)/4}^{\sup_{B_R(x_0)}u}\frac{dt}{\sqrt{F(t)}} \geq cR.
$$
\end{proposition}

\begin{proof}
With no loss of generality, we may assume $x_0 = 0$. For $k \in \{0, \cdots, K\}$, we select $R_k \in (0,R]$ such that $R_0 = 0$, $R_K = R$ and
$$
    M_k = 2M_{k-1} \quad \text{for} \quad 1 \leq k \leq K-1,
$$
and $M_K \leq 2 M_{K-1}$, where $M_k \coloneqq \sup_{B_{R_k}}u$, and $M_0 \coloneqq u(0)$. Indeed, this is always possible by the following reasoning: if $M_K \leq 2u(0)$, then $K=1$ and we are done. If $M_K > 2u(0)$, then, by continuity, there should be $R_1 \in (0,R)$ such that $M_1 = 2u(0)$. In the latter case, if $M_K \leq 2M_1$, then we stop. Otherwise, there should be $R_2 \in (R_1,R)$ such that $M_2 = 2M_1$. We continue doing so until we reach $R_{K-1}$ such that $M_K \leq 2M_{K-1}$. Observe that $K$ is finite, but it depends on the modulus of continuity of $u$. We will see that the estimate does not depend on this number $K$.

Take $x_k \in B_{R_k}$ such that $u(x_k) = M_k$ for every $k \in \{1,\cdots, K-1\}$. Denoting $r_k \coloneqq R_k - R_{k-1}$, it follows that 
$$
    \sup_{B_{r_k}(x_{k-1})}u \leq M_k = 2M_{k-1} = 2u(x_{k-1}),
$$
where the inequality is a consequence of the inclusion $B_{r_k}(x_{k-1}) \subset B_{R_k}$. By Lemma~\ref{est from below} we deduce
$$
    \frac{M_{k-1}}{\sqrt{F(M_{k-1}/2)}} \geq cr_k.
$$
By monotonicity, we obtain
$$
    \int_{M_{k-3}}^{M_{k-2}} \frac{dt}{\sqrt{F(t)}} \geq \frac{M_{k-3}-M_{k-2}}{\sqrt{\Tilde{F}\left(M_{k-2}\right)}} = \frac{1}{4} \frac{M_{k-1}}{\sqrt{F(M_{k-1}/2)}} \geq cr_k.
$$
Summing the above inequality from $k=1$ to $k=K$, we are led to
$$
    \sum_{k=1}^{K} \int_{M_{k-3}}^{M_{k-2}} \frac{dt}{\sqrt{F(t)}} \geq c\sum_{k=1}^{K} r_k = cR,
$$
where we have defined $M_{-1} \coloneqq M_0/2$ and $M_{-2} \coloneqq M_0/4$. The proposition is proved once we realize that
$$
    \int_{u(0)/4}^{\sup_{B_R}u} \frac{dt}{\sqrt{F(t)}} \geq \sum_{k=1}^{K} \int_{M_{k-3}}^{M_{k-2}} \frac{dt}{\sqrt{F(t)}}.
$$
\end{proof}

We next prove our extension of the Keller-Osserman theorem. 

\begin{corollary}\label{keller osserman}
Assume that there exists a nontrivial nonnegative entire solution to
$$
    \mathcal{L}[u] \geq f(u) \quad \text{in} \quad \mathbb{R}^n.
$$
Assume 
$$
    b_1 \in L^{q}_{ul}(\mathbb{R}^n), b_2 \in L^{q}_{ul}(\mathbb{R}^n), \quad \text{and} \quad c \in L^{q/2}_{ul}(\mathbb{R}^n).
$$
Then,
\begin{equation}\label{keller-osserman}
    \int_{1}^{\infty} \frac{dt}{\sqrt{F(t)}} = \infty.
\end{equation}
\end{corollary}

\begin{proof}
Let $\eta>0$ be the quantity from \eqref{coefs universally small}, $c>0$ be the universal constant from Proposition \ref{integral lower bound} and $R>1$ be picked up such that
$$
    cR - 2R^{\frac{n}{q}}\sqrt{\frac{4|B_2|^{\frac{2}{q}}}{\eta}} \geq 2^{-1}c,
$$
which is possible since $q > n$. For $r_0>0$ small such that
$$
    r_0^{2 - \frac{2n}{q}}\left(\|b_1\|_{L^q_{ul}} + \|b_2\|_{L^q_{ul}} + \|c\|_{L^{q/2}_{ul}}  \right) < \eta/2,
$$
we define $v(x) \coloneqq u(r_0 x)$. This new function solves
$$
    \mathcal{L}_0[v] \geq f_0(v),
$$
where $f_0(t) \coloneqq r_0^2 f(t)$, and
$$
    \mathcal{L}_0[v] = \mbox{div}(A_0(x)Dv + b_1^{r_0}(x)v) + b_2^{r_0}(x)\cdot Dv + c_0(x)v,
$$
with $b_1^{r_0}(x) \coloneqq r_0 b_1(r_0 x)$, $b_2^{r_0}(x) \coloneqq r_0 b_1(r_0 x)$ and $c_0(x) \coloneqq r_0^2c(r_0 x)$. In this setting, we have that for any $x_0 \in \mathbb{R}^n$
$$
    \mathcal{S}(x_0) \coloneqq \|b_1^{r_0}\|_{L^{q}(B_{2R}(x_0))} + \|b_2^{r_0}\|_{L^{q}(B_{2R}(x_0))} + \|c_0\|_{L^{q/2}(B_{2R}(x_0))}
$$
satisfy
\begin{eqnarray*}
    \mathcal{S}(x_0) & = & r_0^{1-\frac{n}{q}}\|b_1\|_{L^{q}(B_{2Rr_0}(r_0x_0))} + r_0^{1-\frac{n}{q}}\|b_2\|_{L^{q}(B_{2Rr_0}(r_0x_0))} + r_0^{2-\frac{2n}{q}}\|c\|_{L^{q/2}(B_{2Rr_0}(r_0x_0))}\\
    & \leq & r_0^{2-\frac{2n}{q}}\left(\|b_1\|_{L^{q}_{ul}}+\|b_2\|_{L^{q}_{ul}}+\|c\|_{L^{q/2}_{ul}}\right) \leq \eta/2,
\end{eqnarray*}
by the choice of $r_0$. Pick $\varepsilon>0$ such that 
$$
    \varepsilon |B_{2R}|^{\frac{2}{q}} = \eta/2,
$$
and notice $v$ is also a subsolution to
$$
   \overline{\mathcal{L}}[v] \coloneqq \left(\mathcal{L}_0+\varepsilon\right)[v] \geq \overline{f}(v),
$$
where $\overline{f}(t) = f_0(t) + \varepsilon t$. We argue now that
$$
    \int_{1}^\infty \frac{dt}{\sqrt{\overline{F}(t)}} = \infty,
$$
from which the corollary follows, since 
$$
    F(t) \coloneqq \int_{0}^tf(s)\,ds \lesssim \int_{0}^t\overline{f}(s)\,ds \eqqcolon \overline{F}(t).
$$
We claim that
$$
	\int_{u(0)/4}^{\sup_{B_{Rk}(0)}u}\frac{dt}{\sqrt{\overline{F}(t)}} \geq 2^{-1}ck, \quad \text{for every }\, k \in \mathbb{N},
$$
where
$$
    \overline{F}(t) \coloneqq \int_{0}^t\overline{f}(s)\,ds.
$$
We prove it by induction. For the case $k=1$, it follows from Proposition \ref{integral lower bound} that
$$
    \int_{u(0)/4}^{\sup_{B_R}u}\frac{dt}{\sqrt{\overline{F}(t)}} \geq cR,
$$
where we used that $R>1$. Now, assume it holds for $k$ and let us prove it also holds for $k+1$. Let $\overline{x}$ be the point such that $u(\overline{x}) = \sup_{B_{Rk}(0)} u$. Since $u(0) > 0$, it also follows that $u(\overline{x}) > 0$. We can then apply Proposition \ref{integral lower bound} in $B_{2R}(\overline{x})$ to obtain
$$
	\int_{u(\overline{x})/4}^{\sup_{B_R(\overline{x})}u}\frac{dt}{\sqrt{\overline{F}(t)}} \geq cR.
$$
Therefore
\begin{eqnarray*}
	\int_{u(0)/4}^{\sup_{B_{R(k+1)}(0)}u} \frac{dt}{\sqrt{\overline{F}(t)}} & = & \int_{u(0)/4}^{u(\overline{x})} \frac{dt}{\sqrt{\overline{F}(t)}} + \int_{u(\overline{x})}^{\sup_{B_{R(k+1)}(0)}u}\frac{dt}{\sqrt{\overline{F}(t)}}\\
	& \geq & \int_{u(0)/4}^{\sup_{B_{Rk}(0)} u} \frac{dt}{\sqrt{\overline{F}(t)}} + \int_{u(\overline{x})}^{\sup_{B_R(\overline{x})}u}\frac{dt}{\sqrt{\overline{F}(t)}}\\
	& \geq & 2^{-1}ck + \int_{u(\overline{x})}^{\sup_{B_{R}(\overline{x})}u}\frac{dt}{\sqrt{\overline{F}(t)}},
\end{eqnarray*}
where we used that $\sup_{B_R(\overline{x})}u \leq \sup_{B_{R(k+1)}(0)}u$, which comes from the fact that $B_R(\overline{x}) \subset B_{R(k+1)}(0)$. Now, we write
$$
	\int_{u(\overline{x})}^{\sup_{B_{R}(\overline{x})}u}\frac{dt}{\sqrt{\overline{F}(t)}} = \int_{u(\overline{x})/4}^{\sup_{B_{R}(\overline{x})}u}\frac{dt}{\sqrt{\overline{F}(t)}} - \int_{u(\overline{x})/4}^{u(\overline{x})}\frac{dt}{\sqrt{\overline{F}(t)}}.
$$
Observe that $\overline{f}(t) \geq \varepsilon t$, and so
$$
	\overline{F}(t) \geq \frac{\varepsilon}{2}t^2,
$$
and so
$$
 \int_{u(\overline{x})/4}^{u(\overline{x})}\frac{dt}{\sqrt{\overline{F}(t)}} \leq \sqrt{\frac{2}{\varepsilon}}\int_{u(\overline{x})/4}^{u(\overline{x})}\frac{dt}{t} \leq 2\sqrt{\frac{2}{\varepsilon}}.
$$
This implies that
$$
	\int_{u(0)/4}^{\sup_{B_{R(k+1)}(0)}u}\frac{dt}{\sqrt{\overline{F}(t)}} \geq 2^{-1}ck + cR  - 2\sqrt{\frac{2}{\varepsilon}}.
$$
By the choice of $R$, we have
\begin{eqnarray*}
   cR - 2\sqrt{\frac{2}{\varepsilon}} \geq 2^{-1}c
\end{eqnarray*}
from which the claim follows. Passing to the limit as $k \to \infty$ in this claim, we have
$$
    \int_{1}^\infty\frac{dt}{\sqrt{\overline{F}(t)}} = \infty.
$$
\end{proof}

\subsection{Proof of generalized Harnack inequality}\label{main thm-subsct}

Now we are ready to deliver the

\begin{proof}[Proof of Theorem \ref{main theorem}]
We may assume $u$ is smooth; the general case follows by standard approximation. Replace $u$ by $v(x) \coloneqq u(x_0 + rx)$, for $x_0 \in B_1$. As above, for $r$ small enough, this function solves a PDE whose coefficients have norms smaller than any a priori given constant, with right-hand side $\tilde{f} = r^2 f$.

Recalling the notation from Proposition \ref{ODE for level sets}, let $t_0$ be defined as
$$
    t_0 \coloneqq \sup\left\{t>0 \colon \eta(t) \leq \frac{|B_2|}{2} \right\}.
$$
The proof has two steps. We will first show that there exists a constant $C_0 = C_0(n,\lambda,\Lambda,\|b_1\|_q,\|b_2\|_q,\|c\|_{q/2})>0$ such that
$$
    \int_{\inf_{B_1} v}^{t_0}\frac{ds}{\sqrt{\tilde{F}(s)}+s} \leq C_0.
$$
To do so, we argue as follows: for $\delta_0$ as in Lemma \ref{sublevel measure decay}, there is a constant $C$ depending on $\delta_0$ such that either
\begin{equation}\label{integral from zero to t0}
    \int_{0}^{t_0}\frac{ds}{\sqrt{\tilde{F}(s)}+s}\leq C,
\end{equation}
or there exists $t_\delta \in (0,t_0)$ such that
\begin{equation}\label{t small}
    \eta(t_\delta) \leq \delta_0 \left(\frac{t_\delta}{\sqrt{\tilde{F}(t_\delta)} + t_\delta} \right)^\frac{q n}{q-n} \quad \mbox{and} \quad \int_{t_\delta}^{t_0} \frac{ds}{\sqrt{\tilde{F}(s)}+s}\leq C.
\end{equation}
 We assume this is not true. Then
$$
    \int_{0}^{t_0}\frac{ds}{\sqrt{\tilde{F}(s)}+s} > C,
$$
for any given $C$, and there holds
$$
    \eta(t) \geq \delta_0 \left(\frac{t}{\sqrt{\tilde{F}(t)} + t} \right)^\frac{q n}{q-n} \quad \mbox{for every} \quad t \in (0,t_0),
$$
or
$$
    \int_{t}^{t_0} \frac{ds}{\sqrt{\tilde{F}(s)}+s} > C \quad \mbox{for every} \quad t \in (0,t_0).
$$
In the latter case, if we take $t = t_0/2$, we have
$$
    C < \int_{t_0/2}^{t_0} \frac{ds}{\sqrt{\tilde{F}(s)}+s} \leq \ln(t_0) - \ln(t_0/2) = \ln(2),
$$
which is a contradiction for $C>0$ large. If, on the other hand, we have
$$
    \eta(t) \geq \delta_0 \left(\frac{t}{\sqrt{\tilde{F}(t)} + t} \right)^\frac{q n}{q-n} \quad \mbox{for every} \quad t \in (0,t_0),
$$
then
$$
    \frac{1}{t} \geq \delta_0^{\frac{q-n}{q n}}\eta(t)^{-\frac{q-n}{q n}}\frac{1}{\sqrt{\tilde{F}(t)}+t},
$$
and so, by Proposition \ref{ODE for level sets}, we have
$$
    \eta'(t) \geq c\, \min\left\{\frac{1}{\sqrt{\tilde{F}(t)}}\eta(t)^{\frac{n-1}{n}}, \frac{1}{t}\eta(t) \right\} \geq c\frac{1}{\sqrt{\tilde{F}(t)}+t}\min\left\{\eta(t)^{1 - \frac{1}{n}}, \eta(t)^{1-\frac{q-n}{q n}} \right\}.
$$
Observe that since
$$
    1 - \frac{1}{n} < 1 - \frac{q-n}{q n},
$$
and $\eta(t) \leq \frac{|B_2|}{2}$, we have 
$$
    \min\left\{\eta(t)^{1 - \frac{1}{n}}, \eta(t)^{1-\frac{q-n}{q n}} \right\} = \eta(t)^{1-\frac{q-n}{q n}},
$$
and so
$$
    \eta'(t) \geq c\, \frac{1}{\sqrt{\tilde{F}(t)}+t} \eta(t)^{1-\frac{q-n}{q n}}.
$$
But this is equivalent to saying that
$$
    \frac{d}{dt}\left(\eta(t)^{\frac{q-n}{q n}} \right) \geq c_1 \frac{1}{\sqrt{\tilde{F}(t)}+t},
$$
which implies, after integrating from $t$ to $t_0$,
$$
   |B_2|^{\frac{q-n}{q n}} \geq c_1 \int_{t}^{t_0}\frac{1}{\sqrt{\tilde{F}(s)}+s}ds \quad \mbox{for every} \quad t \in (0,t_0).
$$
By Fatou's lemma
$$
    \int_{0}^{t_0} \frac{ds}  {\sqrt{\tilde{F}(s)}+s} \leq \liminf_{t \to 0^+} \int_{t}^{t_0} \frac{ds}{\sqrt{\tilde{F}(s)}+s},
$$
which is a contradiction with
$$
    \int_{0}^{t_0}\frac{ds}{\sqrt{\tilde{F}(s)}+s} > C,
$$
if $C$ is large enough again. This implies that \eqref{t small} should hold for some $t_\delta$. But then, Lemma \ref{sublevel measure decay} implies that
$$
    \inf_{B_1} v \geq t_\delta/2,
$$
and so
\begin{eqnarray*}
    \int_{\inf_{B_1}v}^{t_0}\frac{ds}{\sqrt{\tilde{F}(s)}+s} &\leq& \int_{\frac{t_\delta}{2}}^{t_\delta} + \int_{t_\delta}^{t_0}\frac{ds}{\sqrt{\tilde{F}(s)}+s}\\
    & \leq & \int_{\frac{t_\delta}{2}}^{t_\delta} \frac{ds}{s} + \int_{t_\delta}^{t_0}\frac{ds}{\sqrt{\tilde{F}(s)}+s} \leq C_0,
\end{eqnarray*}
for some universal constant $C_0$. To finish the proof, it remains to show that
$$
    \int_{t_0}^{\sup_{B_1}v}\frac{ds}{\sqrt{\tilde{F}(s)}+s} \leq C_0.
$$
For a given small $\epsilon>0$ to be chosen later, we let $t_\epsilon \geq t_0$ be the first value such that $\mu(t_\epsilon) \leq \epsilon$, that is
$$
    t_\epsilon \coloneqq \sup\{t>0 \colon \mu(t)>\epsilon \}
$$
(recall $\mu(t) = |A_t \cap B_2|$). Notice that we must have $t_\epsilon > t_0$ since by the very definition of $t_0$ we have $\eta(t_0) = |B_2|/2$, and so, by taking into account that $\eta(t_0) + \mu(t_0) = |B_2|$, we have $\mu(t_0) = |B_2|/2 > \epsilon$, as long as $\epsilon$ is small enough. This implies that $\mu(t) > \epsilon$ for $t \in [t_0,t_\epsilon)$, and so, by Proposition \ref{ODE for level sets}, we have
$$
    -\mu'(t) \geq c \frac{\mu^{1 - \frac{1}{n}}(t)}{\sqrt{\tilde{F}(t)}+t},
$$
for some constant $c$ that depends on $\epsilon$. This is equivalent to saying that
$$
    -\frac{d}{dt}\left(\mu^{\frac{1}{n}}(t) \right) \geq \frac{c}{\sqrt{\tilde{F}(t)}+t},
$$
which leads, by integration from $t_0$ to $t_\epsilon$, to
$$
    \int_{t_0}^{t_\epsilon} \frac{ds}{\sqrt{\tilde{F}(s)}+s} \leq C.
$$
The proof now is done once we show that for $\epsilon$ small enough we have $\sup_{B_1}v \leq 2t_\epsilon$. Indeed, if this is true, then by using what we have just proved, we obtain
$$
    \int_{t_0}^{\sup_{B_1} v} \frac{ds}{\sqrt{\tilde{F}(s)}+s} \leq \int_{t_0}^{t_\epsilon} \frac{ds}{\sqrt{\tilde{F}(s)}+s} + \int_{t_\epsilon}^{2t_\epsilon} \frac{ds}{s} \leq C.
$$
Let us assume for contradiction that $\sup_{B_1} v > 2t_\epsilon$. Recall that by the definition of $t_\epsilon$, we have
$$
    |\{v \geq t_\epsilon\} \cap B_2| \leq \epsilon.
$$
Now, consider $x_0$ and $y_0$ such that $v(x_0) = \sup_{B_1} v$ and $v(y_0) = \sup_{B_{3/2}}v$. Notice that
$$
    B_{5/4 - |x_0|}(x_0) \subset B_{5/4} \quad \text{and} \quad B_{7/4 - |y_0|}(y_0) \subset B_{7/4}.
$$
In particular, by definition of $t_\epsilon$, we have
$$
\max\left\{\left|\left\{v \geq v(x_0)/2 \right\}\cap B_{5/4 - |x_0|}(x_0)\right|, \left|\left\{v \geq v(y_0)/2 \right\}\cap B_{7/4 - |y_0|}(y_0)\right|\right\} \leq \epsilon,
$$
and so, by Lemma \ref{local max princ consequence}, we get
\begin{equation}\label{supremum relations}
    \sup_{B_{5/4}}v \geq 4\sup_{B_1}v \quad \text{and} \quad \sup_{B_{7/4}}v \geq 4\sup_{B_{3/2}}v.
\end{equation}
Consider a sequence of radii $1 = R_0 < R_1 < \cdots $, so that $M_j = \sup_{B_{R_j}}u$ satisfies the relation
$$
    M_j = 2M_{j-1} \quad \text{for} \quad j \geq 1.
$$
We do it until we find $R_K \geq 3/2$. By \eqref{supremum relations}, we obtain that $K\ge3$ and $R_K \leq 7/4$, which ensures it is well-defined. For $j \geq 1$, we define $r_j \coloneqq R_j - R_{j-1}$. We claim that there exists a small $c>0$ such that
$$
    \mu^{\frac{1}{n}}(t) \geq c r_j \quad \text{for} \quad t \in (M_{j-3},M_{j-2}],
$$
for $j \geq 3$. Since $\mu$ is monotone, it is enough to prove this for $t = M_{j-2}$. We assume the contrary to get
$$
    \mu\left(M_{j-2}\right) \leq \epsilon_1 \left|B_{r_j}\right|.
$$
By the construction of the sequence $M_j$, this is equivalent to
$$
    \left|\left\{v \geq M_{j-1}/2 \right\}\cap B_2\right| \leq \epsilon_1 \left|B_{r_j}\right|.
$$
If $x_{j-1} \in B_{R_{j-1}}$ is such that $v\left(x_{j-1}\right) = M_{j-1}$, then by Lemma \ref{local max princ consequence} we have
$$
    \sup_{B_{r_j}\left(x_{j-1}\right)} v \geq 4M_{j-1}.
$$
Taking into account the inclusion $B_{r_j}\left(x_{j-1}\right) \subset B_{R_j}$, it follows that $M_j \geq \sup_{B_{r_j}\left(x_{j-1}\right)} v$. But $M_j = 2M_{j-1}$, leading to
$$
    M_j \geq 4M_{j-1} = 2M_j,
$$
which is a contradiction. Now, since
$$
    \sup_{B_{r_j}\left(x_{j-1}\right)} v \leq M_j = 2M_{j-1} = 2v\left(x_{j-1}\right),
$$
we can use Lemma \ref{est from below} to obtain
$$
    \frac{M_{j-1}}{\sqrt{F(M_{j-1}/2)}} \geq cr_{j},
$$
and by the monotonicity of $F$
$$
    \frac{t}{\sqrt{F(t)}} \geq cr_{j} \quad \text{for} \quad t \in (M_{j-3},M_{j-2}].
$$
We now use Proposition \ref{ODE for level sets} to obtain
$$
    -\mu'(t) \geq \frac{c}{t} \mu^{1 - \frac{1}{n}}(t)r_j, 
$$
which is equivalent to
$$
    -\frac{d}{dt}\left(\mu^{\frac{1}{n}}(t)\right) \geq \frac{c}{t}r_j.
$$
Integrating over $(M_{j-3},M_{j-2})$ and using the relation $M_{j-2} = 2M_{j-3}$, we conclude that
$$
    \mu^{\frac{1}{n}}(M_{j-3}) - \mu^{\frac{1}{n}}(M_{j-2}) \geq cr_j.
$$
Summing the inequality in $j \geq 3$, we have
$$
    \mu^{\frac{1}{n}}\left(\sup_{B_1} v\right) \geq c \sum_{j=3}^K r_j. 
$$
But observe that since $R_K \geq 3/2$, we have
$$
    \sum_{j=1}^K r_j = R_K - R_0 \geq 1/2.
$$
Moreover, since $R_2 \leq 5/4$, we have $r_1 + r_2 \leq 1/4$, which implies that
$$
    \sum_{j=3}^K r_j = \sum_{j=1}^K r_j - (r_1 + r_2) \geq 1/4,
$$
and so
$$
   \mu^{\frac{1}{n}}\left(\sup_{B_1} v\right) \geq c/4 
$$
However we are assuming $\sup_{B_1} v > 2t_{\epsilon}$, and so by the definition of $t_\epsilon$ we have
$$
    4^{-n}c^n \leq \mu\left(\sup_{B_1} v\right) \leq \mu(t_\epsilon) \leq \epsilon,
$$
which is a contradiction for $\epsilon$ small. 

This proves the generalized Harnack inequality for $v$, which, scaling back to $u$, looks like
\begin{equation}\label{local gen harnack}
    \int_{\inf_{B_r(x_0)}u}^{\sup_{B_r(x_0)}u}\frac{ds}{r\sqrt{F(s)} + s} \leq C,
\end{equation}
for some constant $C>0$, and any $x_0 \in B_1$. To extend it to the whole $B_1$, we proceed via a Harnack chain argument. Let $x_{min}$ and $x_{max}$ be the points where the infimum and supremum are attained in $B_1$. Consider $\{B_r(x_i)\}$ be a family of balls such that $B_r(x_i)\cap B_r(x_{i+1})\not = \emptyset$ for every $i \in \{0,1\cdots,k\}$ for some $k$ that depends only on dimension and $r$, with $x_0 = x_{min}$ and $x_{max} = x_k$. First, selecting $y_1 \in B_{r}(x_0)\cap B_r(x_{1})$, we obtain
\begin{eqnarray*}
   \int_{\inf_{B_1}u}^{\sup_{B_1}u}\frac{ds}{\sqrt{F(s)} + s} &\leq& \int_{\inf_{B_1}u}^{u(y_1)}\frac{ds}{\sqrt{F(s)} + s} + \int_{u(y_1)}^{\sup_{B_1}u}\frac{ds}{\sqrt{F(s)} + s}\\ 
   & \leq & \int_{\inf_{B_r(x_0)}u}^{\sup_{B_r(x_0)}u}\frac{ds}{r\sqrt{F(s)} + s} + \int_{\inf_{B_r(x_1)}u}^{\sup_{B_1}u}\frac{ds}{r\sqrt{F(s)} + s}.
\end{eqnarray*}
Now we select $y_2 \in B_r(x_1)\cap B_r(x_2)$ and repeat the same for the second integral. We continue doing so until reaching the last ball in the family of balls. Therefore
$$
    \int_{\inf_{B_1}u}^{\sup_{B_1}u}\frac{ds}{\sqrt{F(s)} + s} \leq \sum_{i=0}^{k-1} \int_{\inf_{B_r(x_i)}u}^{\sup_{B_r(x_{i})}u}\frac{ds}{r\sqrt{F(s)}+s}.
$$
Each one of the integrals in the sum is controlled by \eqref{local gen harnack}, leading to
$$
   \int_{\inf_{B_1}u}^{\sup_{B_1}u}\frac{ds}{\sqrt{F(s)} + s} \leq kC.
$$
Since $k$ is universal, the theorem is proved.
\end{proof}

\subsection{Proof of the strong maximum principle}\label{max princ-subsct}

\begin{proof}
Assume for contradiction that $u \in H^1(\Omega)$ is a nontrivial supersolution to
$$
    \mathcal{L}[u] \leq f(u) \quad \text{in} \quad \Omega.
$$
Arguing as in Step 1 of \cite{SS21}*{Proof of Theorem 1.1}, we can find  a ball
$B^\prime\subset\subset\Omega$ such that $\mathrm{essinf}_{B^\prime}u=0$ but the trace of $u$ on $\partial B^\prime$ does not vanish identically.

We may assume $f\ge 0$ without loss of generality (replace $f$ by $f^+$). As explained in Step 1 of \cite{SS21}*{Proof of Theorem 1.1}, since $0\le u$ are respectively a sub- and a supersolution, the generalized sub- and super-solution method yields a {\it solution} of $\mathcal{L}[\tilde u]=f(\tilde u)$ in $B^\prime$, such that $0\le \tilde u\le u$  in $B^\prime$ and $\tilde u = {\min(u,1)}$ on $\partial B^\prime$ in the trace sense (the latter is preventing $\tilde u$ from vanishing identically). In addition, since $\tilde u\ge 0$ satisfies $\mathcal{L}[\tilde u]\ge 0$ in $B'$ and $\tilde u \le 1$ on $\partial B^\prime$, we can apply  \cite[Theorem 8.15]{GT} to $\mathcal{L}[\tilde u-1]\ge \div(b_1)+c$ and deduce that 
 $\tilde u\in L^\infty(B')$ and $f(\tilde u)\in L^\infty(B')$, so $\tilde u$ is (H\"older) continuous by the De Giorgi-Moser theory (\cite[Theorem 8.15]{GT}). By assumption, since $\mathrm{essinf}_{B^\prime}u=0$, it follows that $\min_{B'} \Tilde{u} = 0$. This would imply, by Theorem \ref{main theorem}, that
$$
    \int_{0}^{\max_{B'}\Tilde{u}}\frac{dt}{\sqrt{F(t)}+t} \leq C,
$$
for some universal constant $C$. But this contradicts the assumption of Theorem \ref{maximum principle}.
\end{proof}

\bigskip

{\small \noindent{\bf Acknowledgments.} A.S. is supported by the King Abdullah University of Science and Technology (KAUST). A.S. acknowledges the Department of Mathematics at PUC-Rio, where most of this research was conducted.}

\bigskip

%%%%%%%%%%%%%%%%%%%% REFERENCES

\end{document}